\documentclass{article}

\usepackage[utf8]{inputenc}
\usepackage{fullpage}

\usepackage{amsmath, amsthm, amssymb}
\usepackage{thmtools, thm-restate}
\usepackage{dsfont}
\usepackage{mathtools}
\mathtoolsset{showonlyrefs=true} % Number equations only if referrenced in the text
\usepackage{graphicx}    % Load graphicx and not graphicx so that pdflatex works properly. No need to load epstopdf: graphicx converts eps to pdf automatically
\usepackage{subfigure}
\usepackage{color}
\usepackage{xcolor}
\definecolor{BlueMAP}{rgb}{0.27,0.45,0.77}

\usepackage[breakable]{tcolorbox}
\newtcolorbox{boxblue}[1]{breakable,colback=white,colframe=BlueMAP,fonttitle=\bfseries,title=#1,left=4pt,right=4pt,top=4pt,bottom=6pt}
\newtcolorbox{boxgreen}[1]{breakable,colback=white,colframe=red,fonttitle=\bfseries,title=#1,left=4pt,right=4pt,top=4pt,bottom=6pt}
\newtcolorbox{partie}{colback=white,colframe=BlueMAP,fonttitle=\bfseries}

\definecolor{darkolivegreen}{rgb}{0.33, 0.42, 0.18}
\definecolor{celestialblue}{rgb}{0.29, 0.59, 0.82}

\usepackage[colorlinks,
            linkcolor=celestialblue,
            citecolor=celestialblue,
            urlcolor=celestialblue]{hyperref}

\hypersetup{
pdftitle={GLM2020},
pdfauthor={Olga Mula, Damiano Lombardi and Felipe Galarce}, 
}

\usepackage{algorithm}
\usepackage{algpseudocode}
\usepackage{soul}

\newtheorem{theorem}{Theorem}[section]

\newtheorem{remark}[theorem]{Remark}

\theoremstyle{definition}
\newtheorem{definition}{Definition}
\numberwithin{equation}{section} % Number eqs with section-number

\newcommand{\cB}{\ensuremath{\mathcal{B}}}

\newcommand{\cI}{\ensuremath{\mathcal{I}}}

\newcommand{\cL}{\ensuremath{\mathcal{L}}}
\newcommand{\cM}{\ensuremath{\mathcal{M}}}
\newcommand{\cN}{\ensuremath{\mathcal{N}}}

\newcommand{\cP}{\ensuremath{\mathcal{P}}}

\newcommand{\cS}{\ensuremath{\mathcal{S}}}

\newcommand{\cW}{\ensuremath{\mathcal{W}}}

\newcommand{\bE}{\ensuremath{\mathbb{E}}}

\newcommand{\bG}{\ensuremath{\mathbb{G}}}

\newcommand{\bN}{\ensuremath{\mathbb{N}}}

\newcommand{\bP}{\ensuremath{\mathbb{P}}}

\newcommand{\bR}{\ensuremath{\mathbb{R}}}
\newcommand{\bS}{\ensuremath{\mathbb{S}}}
\newcommand{\bT}{\ensuremath{\mathbb{T}}}

\newcommand{\pbC}{\ensuremath{\textbf{C}}}
\newcommand{\pbD}{\ensuremath{\textbf{D}}}

\newcommand{\pbH}{\ensuremath{\textbf{H}}}
\newcommand{\pbI}{\ensuremath{\textbf{I}}}

\newcommand{\pbV}{\ensuremath{\textbf{V}}}
\newcommand{\pbW}{\ensuremath{\textbf{W}}}
\newcommand{\pbX}{\ensuremath{\textbf{X}}}

\newcommand{\bell}{\ensuremath{\mathbb{\ell}}}

\newcommand{\rG}{\ensuremath{\mathrm{G}}}

\newcommand{\rV}{\ensuremath{\mathrm{V}}}

\newcommand{\rY}{\ensuremath{\mathrm{Y}}}

\def\[{\left[}
\def\]{\right]}
\def\<{\langle}
\def\>{\rangle}
\def\({\left(}
\def\){\right)}
\def\[{\left [}
\def\]{\right]}
\def\({\left(}
\def\){\right)}

\newcommand{\norm}[1]{\Vert #1 \Vert}

\newcommand{\Vn}{\ensuremath{{V_n}}}
\newcommand{\Wm}{\ensuremath{{W_m}}}

\newcommand{\cond}{\; :\;}

\newcommand{\wc}{\ensuremath{\mathrm{wc}}}
\newcommand{\ms}{\ensuremath{\mathrm{ms}}}

\newcommand{\pbdw}{\ensuremath{\text{(pbdw)}}}
\newcommand{\wcpbdw}{\ensuremath{\text{(wc, pbdw)}}}
\newcommand{\mspbdw}{\ensuremath{\text{(ms, pbdw)}}}

\newcommand{\test}{\ensuremath{\text{test}}}

\newcommand{\ord}{\ensuremath{\mathcal{O}}}
\newcommand{\dist}{\operatorname{dist}}
\newcommand{\diag}{\operatorname{diag}}
\newcommand{\charFun}{\ensuremath{\mathds{1}}}
\DeclareMathOperator*{\argmin}{arg\,min}
\DeclareMathOperator*{\arginf}{arg\,inf}

\newcommand{\dx}{\ensuremath{\mathrm dx}}

\newcommand{\dt}{\ensuremath{\mathrm dt}}
\DeclareMathOperator*{\vspan}{span}
\newcommand{\eps}{\ensuremath{\varepsilon}}

\newcommand{\templates}{\ensuremath{\text{templates}}}
\newcommand{\BT}{\texttt{BT}}
\newcommand{\vox}{\text{vox}}

\newcommand{\om}[1]{\textcolor{black}{#1}} % Comments Olga
\newcommand{\nuevo}[1]{#1}
\newcommand{\new}[1]{#1}

\begin{document}

\title{
State Estimation with Model Reduction and Shape Variability. Application to biomedical problems.
}

\author{Felipe Galarce\footnote{\textsl{Corresponding author.}. Weierstrass-Institut f\"ur Angewandte Analysis und Stochastik. Leibniz-Institut im Forschungsverbund Berlin e.~V.}, Damiano Lombardi\footnote{Centre de Recherche INRIA de Paris \& Laboratoire Jacques-Louis Lions, France.}, Olga Mula\footnote{Paris-Dauphine University, PSL Research University, CNRS, UMR 7534, CEREMADE, France.}}

\maketitle

\begin{abstract}
We develop a mathematical and numerical framework to solve state estimation problems for applications that present variations in the shape of the spatial domain. This situation arises typically in a biomedical context where inverse problems are posed on certain organs or portions of the body which inevitably involve morphological variations. If one wants to provide fast reconstruction methods, the algorithms must take into account the geometric variability. We develop and analyze a method which allows to take this variability into account without needing any a priori knowledge on a parametrization of the geometrical variations. For this, we rely on morphometric techniques involving Multidimensional Scaling, and couple them with reconstruction algorithms that make use of linear subspaces pre-computed on a database of geometries. We prove the potential of the method on a synthetic test problem inspired from the reconstruction of blood flows and quantities of medical interest with Doppler ultrasound imaging.\\

\begin{small}
\end{small}

\begin{small}
\end{small}
\end{abstract}

\section{Introduction}

\subsection{\new{Scientific setting and contribution}}
A central task to address numerous applications in science and engineering is the problem of repeatedly evaluating the output of an expensive forward model for many instances of input parameter values. Such settings include the numerical solution of parametric Partial Differential Equations (PDEs) for different values of the parameter, and, more generally, the multiple evaluation of input-output maps defined by computer models. In the case of parametric PDEs, this task is usually known as model order reduction (MOR). To guide the subsequent discussion, consider the prototypical parametric PDE,
$$
\cP(u, y)(x)=0,\quad \forall x \in \Omega
$$
where $\cP$ is a differential operator and $y$ is a vector of parameters that describes physical properties. We assume $y$ can take values from given compact set $\rY\subset \bR^p$. $\Omega\subset \bR^d$ is an open bounded set which denotes the domain of the independent variables in the PDE. It usually refers to space but it is not limited to that meaning, and it could also refer to more elaborate sets of variables such as time, momentum, or other physical variables. For every $y\in \rY$, we assume that the PDE has a unique solution $u=u(y)$ on a  Hilbert space $V(\Omega)$ of real-valued functions on $\Omega$.

The main task of model reduction is to build a fast parameter-to-solution map that accurately approximates the set of PDE solutions
\begin{equation}
%\label{eq:manifold}
\cM(\Omega) \coloneqq \{ u(y) \in V(\Omega) \; :\; y\in\rY \},
\end{equation}
when the parameter $y$ varies in $\rY$. This set is sometimes referred to as the {\em solution manifold}. Classically, the bottom line of most strategies has been based on approximating $\cM(\Omega)$ with linear spaces from $V(\Omega)$, but developing nonlinear approximation strategies to overcome certain known bottlenecks is a subject of very active research (for some examples, we refer to \cite{AZF2012, Welper2017-TSI, ELMV2020, BCDGJP2021}).

Most linear and nonlinear model reduction methods are built for applications when $\Omega$ is a fixed domain, but in numerous situations $\Omega$ can actually vary, and this raises interesting challenges. This question about geometric variability is by far not new in model reduction, and numerous algorithms and works have been proposed to address it. Most existing strategies crucially rely on assuming that the family of domains $\Omega$ is generated by a known (and easy) parametrization. The problem about the geometric variabillity can then be addressed by computing a reduced model on a reference domain, which is then mapped to a new target geometry in a relatively easy manner thanks to the assumption of an explicit parametrization of the geometries. The type of applications in which these developments have mostly been considered are connected to \emph{forward problems} where the value of $y$ and of the geometrical parameters of the geometry are given as an input to approximate $u(y)\in V(\Omega)$.

This paper is a contribution to go beyond this setting on several fronts:
\begin{itemize}
\item \emph{Nonparametric geometries:} We propose a method where the family of geometries is \emph{nonparametric}. In other words, we do not assume a known parametrization of the domains. This is particularly relevant in certain applications such as biomedical problems where the domain is often given by the shape of a certain organ which inevitably presents morphological variations. These variations are difficult to parametrize, and, in fact, information on the family of geometries is often given by a collection of images of each patient's organ. 
\item \emph{Forward and inverse problems:} We explain how the method can be applied not only for the purpose of forward problems, but also for solving inverse state estimation problems in a reasonable amount of time. In the inverse problem setting, we do not get the value of the parameters $y$. Instead, we are given measurement observations (whose nature varies depending on the application). This makes inverse problems ill posed, and it adds an extra layer of difficulty to the problem. Although the developed method is general and could be used with both linear and nonlinear types of methods, the presentation in this paper focuses mostly on linear methods, and results about the benefits of using nonlinear strategies is left for future work. As such, the main method which we use for forward reduced modeling is based on (linear) Principal Component Analysis (PCA), and we adapt the so-called Parametrised Background Data Weak (PBDW) approach for inverse state estimation. This algorithm was proposed in \cite{MPPY2015} and analyzed in subsequent papers such as \cite{MMT2016, BCDDPW2017, Taddei2017, BCMN2018, GMMT2019, CDDFMN2020, CDMN2020}.
\item \emph{Enhanced use of database information:} To work in a nonparametric geometrical setting, our approach requires to have a collection of domains $\Omega$ and solution manifolds $\cM(\Omega)$ with their associated forward reduced models to carry out a learning phase. We then introduce a notion of measuring distances between manifold sets $\cM(\Omega)$ on different domains $\Omega$. This allows us to find the geometry from the database with the closest physics to a new given target geometry. The precomputed reduced model of the closest neighbor is then transported to the target geometry to solve forward or inverse problem queries. The use of a database of geometries is a salient novelty with respect to most approaches, and it allows to better take into account potential variations of the physics that could be induced by domain changes. Working with a database is however not fully new, though. It was proposed in \cite{guibert2014} which is probably the contribution from the literature that shares more points of contact with the present work. We discuss the main novelties and similarities in Section \ref{sec:previous-works}, where we also give an overview of previous works.
\end{itemize}

Our work was motivated by applications in biomedical engineering. The prototypical yet fundamental situation regards the clinical applications in which non-invasive measurements (typically acquired by medical imaging) are exploited in order to infer non-observable mechanical or physiological properties, or to perform state estimation. The time constraints of the clinical applications clearly motivate and justify the use of model order reduction, and inverse state estimation algorithms based on them. However, the inter-patient variability is often very large and manifests itself also in terms of anatomy, hence the geometrical variability. This points towards working with a database of geometries $\Omega$ obtained from previous patient's examinations in order to be able to examine a wide range of patients. On the data set of geometries, we precompute linear (or nonlinear) reduced models, and then re-use this information for the examination in real time of a new given patient which comes with a new morphology.

\subsection{Organization of the paper} At the end of Section \ref{sec:previous-works} we give an overview of previous works on model reduction in variable geometries. In Section \ref{sec:stateEst} we present the context of inverse state estimation and the methods we use in the present work. In Section \ref{sec:strategy} we detail the strategy we adopt in order to deal with variable domain geometries. In Section \ref{sec:error-analysis} we propose an error analysis for the state estimation problem. In section \ref{sec:practical-realization} we describe in detail how the different steps of the procedure are practically implemented and we conclude by presenting a numerical experiment to assess the method. The example takes inspiration from the reconstruction of blood flows with Doppler ultrasound images.

\subsection{Previous works on model reduction with variable geometries}
\label{sec:previous-works}
The topic of working with variable geometries is common to several fields of research such as shape optimisation (\cite{davies2008,maury2017,de2008}), inverse scattering problems (\cite{colton1998}), geometry morphometrics (\cite{bookstein2018,mitteroecker2009}). In the field of model order reduction, the main challenge stems from the fact that, usually, one needs to define in which linear reduced space the computations are done, and this depends on the domain. Numerous works have studied this issue in the literature. A first example is provided in \cite{lovgren2006}, in which a reduced-element method is devised, to take advantage of domain decomposition techniques and adapt to various potentially deformed domains. An important class of methods consists in mapping the domains into a same reference configuration and write the equations in this latter. In \cite{rozza2013}, the authors consider the set of transformations with affine parametrisation and their effect on the inf-sup stability for a reduced-basis formulation of the Stokes problem. In \cite{manzoni2017eff}, the computational domain is deformed thanks to an elastic displacement and the non-affine dependence of the equations on the domain is tackled by using a matrix-DEIM approach. A similar approach is proposed in \cite{dal2019hyper} to efficiently reduce the computational cost of parametrised fluid models. In \cite{chamoin2019}, an isogeometric analysis framework is used to deal with the domain parametrisation and build a reduced-basis method to speed up shape optimisation problems. A similar approach is proposed in \cite{garotta2020}. In \cite{sevilla2020}, the parametrisation of the domain (obtained by considering \new{ non-uniform rational B-splines \cite{piegl1996nurbs}})) is incorporated as extra-coordinate in a Progressive Generalised Decomposition (PGD) method. In \cite{salmoiraghi2018}, a free-form deformation method is coupled to Proper Orthogonal Decomposition (POD, see e.g. \cite{berkooz1993proper, sirovich1987turbulence, rathinam2003new}) in the context of shape optimisation in aerodynamics. In \cite{karatzas2019}, instead of mapping the domains into a common reference configuration, the shifted boundary method is applied to deal with the geometry parametrisation. By doing so, we avoid the changes of coordinate; to deal with the intrinsic non-linearity, the authors propose to use the GNAT method or the gappy-POD. In \cite{akkari2019}, an hyper-reduction framework is used to deal with non-parametrised geometrical variations of the domain in the context of fluid-mechanics. In \cite{karatzas2020}, a reduced-basis formulation is proposed to deal with a cut-FEM embedding method. In \cite{stabile2020}, the reduced-basis functions are defined on an average-deformed configuration in order to speed up finite volume computations for fluid models with variable geometries. In \cite{forti2014}, the authors consider the problem of the parametrisation of interfaces in the context of fluid-structure interaction problems. In \cite{hess2014}, the reduced-basis method is used to efficiently solve the Maxwell equations to speed up the design of semiconductors. In \cite{taddei2020,taddei2020b} and other recent works, the authors consider the problem of registration applied to model reduction: by suitably transforming the domain we can achieve the reduction efficiency. Numerous applications including geometry reduction can take advantage of such techniques. 

To the best of our knowledge, the work which shares more similarities with the present contribution is \cite{guibert2014}. In that work, a set of realistic patient template geometries is built without knowing the underlying, potentially high-dimensional, parametrisation. The authors then construct the reduced-order model on a reference geometry computed as the average of the available templates (in the sense defined in the paper). A set of transformations allow to map fields between the geometries and the average geometry. The two main differences with respect to the present work are the following: we construct a reduced-order method in view of performing the reconstruction given some observable so, instead of constructing an atlas based solely on geometric information, we construct a set of templates based also on the physics of the problem we are considering. In order to solve the state estimation problem in a reduced way, we adapt the Parametrised Background Data Weak approach \cite{MPPY2015}.

\section{Multi-Domain State Estimation: Problem Setting}
\label{sec:stateEst}
In the following, the terms geometry, spatial domain, and shape will be used interchangeably whenever there is no ambiguity. 

\subsection{State estimation on a given domain}
\label{sec:SE-fixed-domain}
Let $\Omega$ be a fixed given domain of $\bR^d$ with dimension $d\geq 1$, and let $V(\Omega)$ be a Hilbert space defined over $\Omega$. The space is endowed with an inner product $\<\cdot, \cdot\>$ and induced norm $\Vert \cdot \Vert$. The choice of $V(\Omega)$ must be relevant for the problem under consideration, and typical options are $L^2$, $H^1$ or some Reproducing Kernel Hilbert Space.

Our goal is to recover an unknown function $u\in V(\Omega)$ from $m$ possibly noisy measurement observations
\begin{equation}
\label{eq:meas}
y_i = \ell_i(u) + \eta_i,\quad i=1,\dots,m,
\end{equation}
where the $\ell_i$ are linearly independent linear forms from $V'(\Omega)$ and the $\eta_i$ are unknown measurement errors. In the following, for the sake of simplicity, we will assume that there is no noise ($\eta_i = 0, \, i=1,\dots,m$) but the main methodology which we develop could easily be extended to deal with noisy measurements. In practical applications, each $\ell_i$ models a sensor device which is used to collect the measurement data $\ell_i(u)$. \new{In the applications which we present in our numerical tests, the observations come in the form of an image and each $\ell_i$ models the response of the system on a given pixel. Figure \ref{fig:doppler-img} illustrates a complete synthetic image.}

We denote by $\omega_i \in V(\Omega)$ the Riesz representers of the $\ell_i$. They are defined via the variational equation
$$
\left< \omega_i, v \right> = \ell_i(v),\quad \forall v \in V(\Omega).
$$
Since the $\ell_i$ are linearly independent in $V'(\Omega)$, so are the $\omega_i$ in $V(\Omega)$ and they span an $m$-dimensional space
$$
\Wm(\Omega)={\rm span}\{\omega_1,\dots,\omega_m\} \subset V(\Omega).
$$

When there is no measurement noise, knowing the observations $y_i=\ell_i(u)$ is equivalent to knowing the orthogonal projection
\begin{equation}
\omega = P_{\Wm(\Omega)} u.
\end{equation}
In this setting, the task of recovering $u$ from the measurement observation $\omega$ can be viewed as building a recovery algorithm
$$
A:\Wm(\Omega)\mapsto V(\Omega)
$$
such that $A(P_{\Wm(\Omega)}u)$ is a good approximation of $u$ in the sense that $\Vert u - A(P_{\Wm(\Omega)}u) \Vert$ is small.

Recovering $u$ from the measurements $P_{\Wm(\Omega)} u$ is a very ill-posed problem since $V(\Omega)$ is generally a space of very high or infinite dimension so, in general, there are infinitely many $v\in V(\Omega)$ such that $P_{\Wm(\Omega)} v = \omega$. It is thus necessary to add some a priori information on $u$ in order to recover the state up to a guaranteed accuracy. In the following, we work in the setting where $u$ is a solution to some parameter-dependent PDE of the general form
\begin{equation*}
\cP(u, y) = 0,
\end{equation*}
where $\cP$ is a differential operator and $y$ is a vector of parameters that describes some physical property and lives in a given set $\rY\subset \bR^p$. For every $y\in \rY$, we assume that the PDE has a unique solution $u=u(y)\in V(\Omega)$. Therefore, our prior on $u$ is that it belongs to the set
\begin{equation}
\label{eq:manifold}
\cM(\Omega) \coloneqq \{ u(y) \in V(\Omega) \; :\; y\in\rY \},
\end{equation}
which is sometimes referred to as the {\em solution manifold}.

\paragraph{Performance Benchmarks:} The quality of a recovery mapping $A$ is usually quantified in two ways:
\begin{itemize}
\item If the sole prior information is that $u$ belongs to the manifold $\cM(\Omega)$, the performance is usually measured by the worst case reconstruction error
\begin{equation}
\label{eq:err-A-wc}
E_{\wc}(A,\cM(\Omega)) = \sup_{u\in\cM(\Omega)} \Vert u - A(P_{\Wm(\Omega)} u) \Vert \, .
\end{equation}
\item In some cases $u$ is described by a probability distribution $p$ on $V(\Omega)$ supported on $\cM(\Omega)$. This distribution is itself induced by a probability distribution on $\rY$ that is assumed to be known. When no information about the distribution is available, usually the uniform distribution is taken. In this Bayesian-type setting, the performance is usually measured in an average sense through the mean-square error
\begin{equation}
\label{eq:err-A-ms}
E^2_{\ms}(A,\cM(\Omega)) = \bE\left( \Vert u - A(P_{\Wm(\Omega)} u) \Vert^2\right) = \int_{V(\Omega)} \Vert u - A(P_{\Wm(\Omega)} u)\Vert^2 dp(u) \, ,
\end{equation}
and it naturally follows that $E_{\ms}(A,\cM(\Omega))\leq E_{\wc}(A,\cM(\Omega)) $.
\end{itemize}

\paragraph{PBDW as our practical algorithm:} In this work, we will reconstruct with the Parametrized-Background Data-Weak algorithm (PBDW, \cite{MPPY2015}). Other choices would of course be possible but the PBDW algorithm is relevant for the following reasons:
\begin{itemize}
\item \textbf{Simplicity and Speed:} It is easily implementable and it provides reconstructions in near-real time.
\item \textbf{Optimality:} It has strong connections with optimal linear reconstruction algorithms as has been studied in \cite{BCDDPW2017, CDDFMN2020}.
\item \textbf{Extensions:} If required, the algorithm can easily be extended to enhance its reconstruction performance (see \cite{CDMN2020, GGLM2021}). In particular, it is shown in \cite{CDMN2020} that piece-wise PBDW reconstruction strategy can deliver near-optimal reconstruction performance. The PBDW algorithm can also be easily adapted to accommodate noisy measurements (see \cite{Taddei2017, GMMT2019}) and some easy to implement extensions to mitigate the model error exist (in the following however, we assume the PDE model is perfect for the sake of simplicity).
\end{itemize}
Since the geometry of $\cM(\Omega)$ is generally complex, optimization tasks posed on $\cM(\Omega)$ are difficult (lack of convexity, high evaluation costs for different parameters). Therefore, instead of working with $\cM(\Omega)$, PBDW works with a linear (or affine) subspace $V_n(\Omega)$ of reduced dimension $n$ which is expected to approximate the solution manifold well in the sense that the approximation error of the manifold,
\begin{equation}
\label{eq:error-manifold}
\delta_n^{(\wc)} \coloneqq \sup_{u\in\cM(\Omega)} \dist(u, V_n(\Omega))  \,
,\quad \text{or} \quad
\delta^{(\ms)}_n \coloneqq \bE\left(  \dist(u, V_n(\Omega))^2\right)^{1/2} \, 
\end{equation}
decays rapidly if we increase the dimension $n$. It has been proven in \cite{CD2015acta} that it is possible to find such hierarchies of spaces $(V_n(\Omega))_{n\geq 1}$ for certain manifolds coming from classes of elliptic and parabolic problems, and numerous strategies have been proposed to build the spaces in practice (see, e.g., \cite{BMPPT2012,RHP2007} for reduced basis techniques and \cite{CD2015acta,CDS2011} for polynomial approximations in the $y$ variable).

Assuming that we are given a \new{linear subspace} $V_n(\Omega)$ with $1\leq n\leq m$, the PBDW algorithm $$A^\pbdw_{m,n}:\Wm \to V(\Omega),$$ gives for any $\omega \in \Wm(\Omega)$ a solution of
\begin{equation}
\label{eq:pbdw-algo}
A^\pbdw_{m,n}(\omega) \in \argmin_{u \in \omega + \new{W_m(\Omega)}^\perp } \dist(u, V_n(\Omega)).
\end{equation}

\new{
For any pair of closed subspaces $(E,F)$ of $V$, we define $\beta(E, F)$ as 
\begin{equation}
\label{eq:infsup}
\beta(E,F):=\inf_{e\in E}\sup_{f\in F}\frac {\<e,f\>}{\|e\|\, \|f\|}=\inf_{e\in E}\frac {\|P_{F} e\|}{\|e\|} \in [0,1] .
\end{equation}
The minimizer of \eqref{eq:pbdw-algo} is unique as soon as $n\leq m$ and $\beta(\Vn(\Omega),\Wm(\Omega))>0$, which is an assumption to which we adhere in the following. In practice, solving problem \eqref{eq:pbdw-algo} boils down to solving a linear least squares minimization problem. We refer, e.g., to \cite[Appendix A]{GLM2021} for details on how to compute it in practice. 
}

We can prove that $A^\pbdw_{m,n}$ is a bounded linear map from $\Wm(\Omega)$ to $\Vn(\Omega)\oplus(W_m(\Omega)\cap V_n(\Omega)^\perp)$. In fact, it is a simple least squares problem whose cost is essentially \new{$mn + n^2$}. Therefore, if the dimension $n$ of the linear subspace is moderate, the reconstruction with \eqref{eq:pbdw-algo} takes place in close to real-time.

\new{The reconstruction error $\Vert u - A^\pbdw_{m,n}(\omega) \Vert$ was first studied in \cite{MPPY2015} and further developed in \cite{BCMN2018}. In the later, the following result can be found: for any $u\in V(\Omega)$ it holds}
\begin{equation}
\Vert u - A^\pbdw_{m,n}(\omega) \Vert
\leq \beta^{-1}(V_n, \Wm) \Vert u - P_{\Vn\oplus(W_m\cap V_n^\perp) } u \Vert
\leq \beta^{-1}(V_n, \Wm) \Vert u - P_{\Vn}u \Vert,
\label{eq:pbdw_bound}
\end{equation}
where we have omitted the dependency of the spaces on $\Omega$ in order not to overload the notation, and we will keep omitting this dependency until the end of this section. Depending on whether $\Vn$ is built to address the worst case or mean square error, the reconstruction performance over the whole manifold $\cM$ is bounded by
\begin{equation}
\label{eq:err-wc-pbdw}
e_{m,n}^\wcpbdw \coloneqq
E_{\wc}(A^\pbdw_{m,n}, \cM) \leq \beta^{-1}(V_n, \Wm) \new{\sup_{u\in \cM}}\dist(u, V_n\oplus(V_n^\perp\cap\Wm)) \leq \beta^{-1}(V_n, \Wm) \, \delta_n^{(\wc)},
\end{equation}
or
\begin{align}
e_{m,n}^\mspbdw \coloneqq
E_{\ms}(A^\pbdw_{m,n}, \cM)
&\leq \beta^{-1}(V_n, \Wm) \bE\left(\dist(u, V_n\oplus(V_n^\perp\cap\Wm))^2\right)^{1/2}  \nonumber\\
&\leq \beta^{-1}(V_n, \Wm) \, \delta_n^{(\ms)} . \label{eq:err-ms-pbdw}
\end{align}
Note that $\beta(\Vn, \Wm)$ can be understood as a stability constant. It can also be interpreted as the cosine of the angle between $\Vn$ and $\Wm$. The error bounds involve the distance of $u$ to the space $V_n\oplus(V_n^\perp\cap\Wm)$ which provides slightly more accuracy than the linear subspace $\Vn$ alone. This term is the reason why it is sometimes said that the method can correct model error to some extent. In the following, to ease the reading we will write errors only with the second type of bounds \eqref{eq:err-ms-pbdw} that do not involve the correction part on $V_n^\perp\cap\Wm$.

An important observation is that for a fixed measurement space $W_m$ (which is the setting in our numerical tests), the error functions
$$
n\mapsto e_{m,n}^\wcpbdw,\quad \text{and} \quad n\mapsto e_{m,n}^\mspbdw,
$$
reach a minimal value for a certain dimension $n^*_\wc$ and $n^*_\ms$ as the dimension $n$ varies from 1 to $m$. This behavior is due to the trade-off between:
\begin{itemize}
\item the improvement of the approximation properties of $V_n$ as $n$ grows ($\delta_n^{(\wc)}$ and $\delta_n^{(\ms)} \to 0$ as $n$ grows)
\item the degradation of the stability of the algorithm, given here by the decrease of $\beta(V_n, \Wm)$ to 0 as $n\to m$. When $n> m$, $\beta(V_n, \Wm)=0$.
\end{itemize}
As a result, the best reconstruction performance with  PBDW is given by
$$
e_{m,n^*_\wc}^\wcpbdw = \min_{1\leq n \leq m} e_{m,n}^\wcpbdw,
\quad\text{or}\quad
e_{m,n^*_\ms}^\mspbdw = \min_{1\leq n \leq m} e_{m,n}^\mspbdw.
$$

\subsection{Obstructions when the spatial domain is not given a priori}
The speed of the above reconstruction algorithm crucially relies on the fact that we have assumed that the spatial domain $\Omega$ is given to us a priori. Thanks to this we can precompute the linear subspaces $V_n(\Omega)$ before the reconstruction takes place, and we only need to solve \eqref{eq:pbdw-algo} during the  reconstruction, which is a computation that can be done in near real-time.  The offline computation of the reduced model should be seen as a training phase, and it can be computationally intensive and time-consuming for complex physical systems.

There are however cases in which we cannot assume that $\Omega$ is given a priori. This situation typically arises in biomedical applications where state estimation needs to be performed on a certain part of the body for different patients which inevitably present morphological variations. In this case, given a new target geometry $\Omega$, one could of course generate $\cM(\Omega)$ and derive a linear subspace $V_n(\Omega)$ but this task would not be feasible in real-time, and the method would no be useful for real time decisions. To avoid this computational bottleneck, we propose a method to quickly build a space $V_n(\Omega)$ by using reduced models which have been pre-computed on a database of template geometries which we suppose to be available offline. The idea consists in finding the best reduced model from the template geometries, and then to transport it to the target geometry $\Omega$. Once this is done, we reconstruct with PBDW on the target geometry. The next section presents the details of our proposed strategy.

\section{Proposed strategy for fast state estimation}
\label{sec:strategy}
We consider a set $\rG$ of spatial domains in $\bR^d$. The set can potentially be infinite. An example for $\rG$ is the set of human carotid arteries or, more generally,  the set of shapes of a certain organ. Our goal is to build a state estimation procedure that is fast for every geometry $\Omega \in \rG$. For this, our approach is based on a learning phase that involves computations on a dataset of available template geometries. We next summarize the main steps. In section \ref{sec:error-analysis} we give an error analysis of the procedure and discuss the main sources of inaccuracy. Some steps involve certain routines which are introduced at an abstract level in this section and in the error analysis. In section \ref{sec:practical-realization}, we explain how we have implemented them in practice, and how our theory justifies certain choices. Note however that since the procedure is general, other constructions can of course be considered for these building blocks.

\vspace{0.3cm}
\begin{boxblue}{\textbf{Training/Learning phase (offline)}}
\begin{itemize}
\item \textbf{Database of Template Geometries:} Gather a family  of $K$ template domains
$$
\rG_\templates = \{ \Omega_1,\dots, \Omega_K\} \subseteq \rG.
$$
This family will serve as a database for our subsequent developments.
\item \textbf{Database of Template Reduced Models:}  For every $\Omega \in \rG_\templates $, similarly as in section \ref{sec:SE-fixed-domain} we consider a parameter-dependent PDE
$$
\cP(u, y) = 0,
$$
where the parameters $y$ take values in $\rY$ and the solution $u(y)$ belongs to a Hilbert space $V(\Omega)$. Note that the differential operator $\cP$ and the parameter domain $\rY$ could vary with the geometry $\Omega$. However, to simplify the presentation, we assume that $\cP$ and $\rY$ are taken identical for all $\Omega \in \rG_\templates$. The set of solutions yields the solution manifold $\cM(\Omega)$ and it describes all the possible physical states of the system under consideration for the given geometry. We summarize the physics by precomputing a template linear subspace $V_n(\Omega)$,
$$
\cM(\Omega) \approx V_n(\Omega),\quad \forall \Omega \in \rG_\templates.
$$
\item \textbf{Transport snapshots and \new{linear subspaces} between geometries:} We need to define a map to transport functions between different geometries
  \begin{align*}
 \tau_{\Omega\to\Omega'}:  V(\Omega) \to V(\Omega'),\quad \forall (\Omega, \Omega') \in \rG \times \rG.
 \end{align*}
We also need to define a map to transport subspaces into subspaces. Since in general the image of a subspace $V_n(\Omega)$ by $\tau_{\Omega\to\Omega'}$ is not necessarily a subspace, we introduce another mapping
\begin{align*}
\widehat \tau_{\Omega\to\Omega'}: \{V_n(\Omega) \subseteq V(\Omega) \} \to \{V_{n'}(\Omega') \subseteq V(\Omega')\},\quad \forall (\Omega, \Omega') \in \rG \times \rG.
 \end{align*} 
\new{We assume in the following that $\widehat{\tau}_{\Omega \rightarrow \Omega'}$ is chosen such that $n' = n$.} Also, for some applications, it will be important that $\tau$ satisfies some physical properties such as mass conservation. We discuss how we have built $\tau$ and $\widehat \tau$ in practice in section \ref{sec:transport-routine}.
 \item \texttt{Best-Template}: For the reconstruction task, we need to identify for each new target geometry $\Omega \in \rG$ which template geometry $\Omega_t \in \rG_\templates $ has the most appropriate linear subspace $V_n(\Omega_t)$ that we have to transport to $\Omega$. For this, we need to build a best template map
\begin{align*}
\BT: \rG &\to \rG_\templates \\
\Omega &\mapsto \Omega^*_t.
\end{align*}
We discuss the different possibilities to build \BT~in section \ref{sec:best-template}.
\end{itemize}
\end{boxblue}
   
\begin{boxgreen}{\textbf{Reconstruction phase (online)}}
We are given a target domain $\Omega \in \rG$, and our goal is to give a fast reconstruction of an unknown function $u\in V(\Omega)$ given $m$ measurement observations $\bell(u) = \left( \ell_i(u)\right)_{i=1}^m$. Note that since $\ell_i \in V'(\Omega)$, the observation space depends on the geometry and $W=W(\Omega)$.
 \begin{itemize}
% \item Voxelize $\Omega$.
 \item If $\Omega \in \rG_\templates$ (the target geometry is in our template dataset), then we simply reconstruct with $A^\pbdw_{n,m}(P_{W(\Omega)} u)$ with the pre-computed linear subspace $V_n(\Omega)$.
 \item If $\Omega \not\in \rG_\templates$:
 \begin{itemize}
 \item We need to find an appropriate linear subspace for the reconstruction. For this, we apply the best-template mapping \BT~and we set
 $$
 \Omega^*_t = \BT(\Omega) \; \in \rG_\templates.
 $$
\item We transport the template linear subspace $V_n(\Omega^*_t)$ to $\Omega$ by applying $\widehat\tau_{\Omega^*_t \to \Omega}$, namely
$$
\widehat{V}_n(\Omega) = \widehat\tau_{\Omega^*_t \to \Omega}( V_n (\Omega^*_t) ).
$$
\item In $\Omega$, we reconstruct with PBDW using $\Wm(\Omega)$ and $\widehat{V}_n(\Omega)$.
 \end{itemize}
 \end{itemize}
 \end{boxgreen}

 \section{Theoretical analysis of the reconstruction error}
 \label{sec:error-analysis}
Suppose we are given a target geometry $\Omega_1\in \rG$ and that we want to reconstruct an unknown function $u\in \cM(\Omega_1)$ from its observations $\ell_i(u),\,i=1,\dots, m$. Suppose further that we fix a geometry $\Omega_0\in \rG_\templates$ and we transport the linear subspace space $V_n(\Omega_0)$ to the target geometry by applying  $\widehat \tau_{0\to1}(V(\Omega_0))$. The goal of this section is to give an error bound on the reconstruction of $u\in \cM(\Omega_1)$ with PBDW and using 
$$
\widehat{V_n}(\Omega_1) = \widehat\tau_{\Omega_0 \to \Omega_1}( V_n (\Omega_0) ),
$$
as a linear subspace on $\Omega_1$. \new{To ease the notation, we will use $\tau_{0\rightarrow 1}$ and $\hat{\tau}_{0\rightarrow 1}$ to denote $\tau_{\Omega_0 \to \Omega_1}$ and $\hat \tau_{\Omega_0 \to \Omega_1}$ respectively.}

The results involve the following notion of Hausdorff distance between compact sets.
\begin{definition}
For any two given compact sets $X$ and $Y$ of a Hilbert space $V$, the Hausdorff distance between $X$ and $Y$ is defined as
$$
d_H(X, Y) \coloneqq \max\{ \sup_{x\in X} \Vert x-P_Y x \Vert_V, \sup_{y\in Y} \Vert y - P_X y \Vert_V \}.
$$
\end{definition}

\subsection{An error bound based on $d_H\left( \tau_{0\to1}\new{ \(\cM(\Omega_0)\)}, \cM(\Omega_1) \right)$}
It is natural to expect that the reconstruction error will be of good quality if:
\begin{itemize}
\item the physical phenomena contained in the target manifold $\cM(\Omega_1) $ are well represented in some sense by the transported manifold $\tau_{0\to1}\new{\(\cM(\Omega_0)\)}$, and if
\item the \new{linear subspace} $V_n(\Omega_0)$ approximates \new{$\cM(\Omega_0)$} with enough accuracy, and its quality is not degraded by the transport to the target geometry.
\end{itemize}
Theorem \ref{thm:bound1} formalises and quantifies this intuition under the following assumptions:
\begin{enumerate}
\item In the template geometry $\Omega_0$, the accuracy of the template linear subspace $V_n(\Omega_0)$ is bounded by
\begin{equation}
\max_{u\in \cM(\Omega_0)} \Vert u - P_{V_n(\Omega)} u \Vert \leq \eps_0,
\tag{H1}
\label{eq:H1}
\end{equation}
for some $\eps_0\geq0$.
\item The Hausdorff distance between $\tau_{0\to1}\left(\cM(\Omega_0)\right)$ and $\cM(\Omega_1)$ is bounded by
\begin{equation}
d_H\left( \tau_{0\to1}\(\cM(\Omega_0)\), \cM(\Omega_1) \right) \leq \eta,
\tag{H2}
\label{eq:H2}
\end{equation}
for some $\eta\geq0$. Note that $d_H\left( \tau_{0\to1} \( \cM(\Omega_0) \), \cM(\Omega_1) \right)$ couples the physics, the geometry and the transport between $\Omega_0$ and $\Omega_1$. The bound on this term expresses the fact that the physics in the target domain $\Omega_1$, expressed via the manifold $\cM(\Omega_1)$, should be well represented when we transport the physics from $\Omega_0$ to $\Omega_1$. The value of $\eta$ could of course be large depending on the type of physics, geometries, and transport.
\item We finally need two technical assumptions on the transport maps $\tau_{0\to 1}$ and $\widehat \tau_{0\to 1}$:
\begin{enumerate}
\item $\tau_{0\to1}:V(\Omega_0)\to V(\Omega_1)$ is Hölder continuous, namely there exists $C>0$ and $\alpha>0$ such that
\begin{equation}
\Vert \tau_{0\to1}(f) - \tau_{0\to1}(g) \Vert_{V(\Omega_1)} \leq C \Vert f - g \Vert^\alpha_{V(\Omega_0)},\quad \forall (f, g)\in V(\Omega_0) \times V(\Omega_0).
\tag{H3}
\label{eq:H3}
\end{equation}
\item There exists $\gamma\geq 0$ such that
\begin{equation}
\sup_{v\in \tau_{0\to 1}(V_n(\Omega_0))} \frac{ \Vert v - P_{\widehat \tau_{0\to1} \new{(V_n(\Omega_0))}} v \Vert_{V(\Omega_1)} }{\Vert v \Vert_{V(\Omega_1)}} \leq \gamma .
\tag{H4}
\label{eq:H4}
\end{equation}
\end{enumerate}
\end{enumerate}

\begin{theorem}
\label{thm:bound1}
Let $\Omega_0\in \rG_\templates$ be a template geometry and let $u \in \cM(\Omega_1)$ be a target function to estimate from the observations $P_{\Wm(\Omega_1)}u$. If we reconstruct with PBDW using
$$
\widehat V_n(\Omega_1) = \widehat \tau_{0\to1} (V_n(\Omega_0)),
$$
then the reconstruction error is bounded by
\begin{equation}
\Vert u - A(P_{\Wm(\Omega_1)} u) \Vert_{V(\Omega_1)} \leq \frac{1}{\beta(\widehat V_n(\Omega_1), W_m(\Omega_1))} \Vert u - P_{\widehat V_n(\Omega_1)} u\Vert_{V(\Omega_1)}.
\label{eq:err-pbdw-Vhat}
\end{equation}
If the assumptions \eqref{eq:H1} to \eqref{eq:H4} hold, then the reconstruction error over the whole manifold $\cM(\Omega_1)$ is bounded by
\begin{equation}
\max_{u \in \cM(\Omega_1)} \Vert u - A(P_W u) \Vert_{V(\Omega_1)} \leq \frac{1}{\beta(\widehat V_n(\Omega_1), W_m(\Omega_1))}(\eta+ \max_{v\in \tau_{0\to1} \(\cM(\Omega_0)\) } \Vert v - P_{\widehat V_n(\Omega_1)} (v)\Vert_{V(\Omega_1)} ).
\label{eq:reconstruction-error-1}
\end{equation}
Suboptimal bounds for  $\max_{v\in \tau_{0\to1}\(\cM(\Omega_0)\) } \Vert v - P_{\widehat V_n(\Omega_1)} (v)\Vert_{V(\Omega_1)}$ are
\begin{align}
\max_{v\in \tau_{0\to1} \new{ \(\cM(\Omega_0)\) } } \Vert v - P_{\widehat V_n(\Omega_1)} (v)\Vert_{V(\Omega_1)}
&\leq
C\eps_0^\alpha+  \max_{v\in \tau_{0\to1} \( P_{V_n(\Omega_0)} \cM(\Omega_0) \)} \Vert v - P_{\widehat V_n(\Omega_1)} v\Vert_{V(\Omega_1)} \label{eq:subopt-1}\\
&\leq   C(\eps_0^\alpha+ \gamma\, \max_{u\in \cM(\Omega_0)}\Vert u \Vert_{V(\Omega_0)}^\alpha), \label{eq:subopt-2}
\end{align}
where the constant $C>0$ is the one given in assumption \eqref{eq:H3}.
\end{theorem}

\begin{proof}
In this proof, all norms will be related to the space $V(\Omega_1)$ defined on the target geometry $\Omega_1$. Let $u\in \cM(\Omega_1)$. By \eqref{eq:pbdw_bound}, we have
\begin{equation}
\Vert u - A(\new{P_{W_m}} u) \Vert \leq \frac{1}{\beta(\widehat V_n(\Omega_1), \new{{W_m}})} \Vert u - P_{\widehat V_n(\Omega_1)} u\Vert,
\label{eq:err-pbdw-proof}
\end{equation}
which is the first inequality of the Theorem. We next bound $\Vert u - P_{\widehat V_n(\Omega_1)} u\Vert$ in terms of quantities in the template geometry $\Omega_0$ and the transport operators $\tau_{0\to 1}$ and $\widehat \tau_{0\to 1}$. For this, let
$$
u_1 \in \arginf_{v\in \tau_{0\to1}(\cM(\Omega_0)) } \Vert u - v \Vert,
$$
and remark that
\begin{equation}
\Vert u - u_1 \Vert \leq d_H(\tau_{0\to1}(\cM(\Omega_0)), \cM(\Omega_1))\leq \eta,
\label{eq:tilde-u}
\end{equation}
by assumption \eqref{eq:H2}.

By the triangle inequality and inequality \eqref{eq:tilde-u},
\begin{align}
\Vert u - P_{\widehat V_n(\Omega_1)} u\Vert
&\leq \Vert u - u_1 - P_{\widehat V_n(\Omega_1)}(u-u_1) \Vert + \Vert u_1 -P_{\widehat V_n(\Omega_1)} u_1 \Vert \\
&\leq \eta + \Vert u_1 -P_{\widehat V_n(\Omega_1)} u_1 \Vert \\
&\leq \eta + \max_{v\in \tau_{0\to1}\(\cM(\Omega_0)\)} \Vert v -P_{\widehat V_n(\Omega_1)} v \Vert,
\label{eq:bound-proj-err}
\end{align}
and the error bound \eqref{eq:reconstruction-error-1} follows by inserting \eqref{eq:bound-proj-err} into \eqref{eq:err-pbdw-proof}.

We next bound $\max_{v\in \tau_{0\to1} \(\cM(\Omega_0)\)} \Vert v -P_{\widehat V_n(\Omega_1)} v \Vert$ as follows. For any $u_1 \in \tau_{0\to1} \(\cM(\Omega_0)\)$, there exists $u_0\in \cM(\Omega_0)$ such that $u_1 = \tau_{0\to1}(u_0)$. Therefore,
\begin{align*}
\Vert u_1 -P_{\widehat V_n(\Omega_1)} u_1 \Vert
& = \Vert \tau_{0\to1}(u_0) -  P_{\widehat V_n(\Omega_1)} \left( \tau_{0\to1}(u_0)\right) \Vert \\
& \leq \Vert \tau_{0\to1}(u_0) -  \tau_{0\to1}(P_{V_n(\Omega_0)}u_0) - P_{\widehat V_n(\Omega_1)}\left[  \tau_{0\to1}(u_0) -  \tau_{0\to1}(P_{V_n(\Omega_0)}u_0) \right] \Vert\\
&\quad+
\Vert \tau_{0\to1}(P_{V_n(\Omega_0)}u_0)  - P_{\widehat V_n(\Omega_1)} \left(\tau_{0\to1}(P_{V_n(\Omega_0)}u_0) \right)\Vert,
\end{align*}
where we have added and subtracted $ \tau_{0\to1}(P_{V_n(\Omega_0)}u_0)\new{-}P_{\widehat V_n(\Omega_1)}\left( \tau_{0\to1}(P_{V_n(\Omega_0)}u_0)\right)$, and applied the triangle inequality. By applying hypotheses \eqref{eq:H3} and \eqref{eq:H1}, we can further bound the above inequality as
\begin{align}
\Vert u_1 -P_{\widehat V_n(\Omega_1)} u_1 \Vert
& \leq \Vert \tau_{0\to1}(u_0) -  \tau_{0\to1}(P_{V_n(\Omega_0)}u_0) \Vert
+ \max_{v\in \tau_{0\to1}\( P_{V_n(\Omega_0)}\cM(\Omega_0) \) }\Vert v - P_{\widehat V_n(\Omega_1)} v\Vert\\
&\leq C\eps_0^\alpha + \max_{v\in \tau_{0\to1} \( P_{V_n(\Omega_0)}\cM(\Omega_0) \) }\Vert v - P_{\widehat V_n(\Omega_1)} v\Vert,
\label{eq:last-ineq}
\end{align}
which yields inequality \eqref{eq:subopt-1}. Inequality \eqref{eq:subopt-2} follows from using \eqref{eq:H3} \new{and \eqref{eq:H4}} to bound  $\max_{v\in \tau_{0\to1}\(P_{V_n(\Omega_0)}\cM(\Omega_0) \) }\Vert v - P_{\widehat V_n(\Omega_1)} v\Vert$ in \eqref{eq:last-ineq}. Note that both inequalities \eqref{eq:subopt-1} and \eqref{eq:subopt-2} are suboptimal due to the construction of the bounds.

%\begin{align}
%\Vert u_1 -P_{\widehat V_n(\Omega_1)} u_1 \Vert
%& = \Vert \tau_{0\to1}(u_0) -  P_{\widehat V_n(\Omega_1)} \left( \tau_{0\to1}(u_0)\right) \Vert \\
%& \leq \Vert \tau_{0\to1}(u_0) -  \tau_{0\to1}(P_{V_n(\Omega_0)}u_0) \Vert
%+ \gamma
%\Vert \tau_{0\to1}(P_{V_n(\Omega_0)}u_0)  \Vert, \qquad \text{by \eqref{eq:H4}} \\
%&\leq C \left( \Vert u_0 - P_{V_n(\Omega_0)}u_0 \Vert^\alpha + \gamma \Vert u_0 \Vert^\alpha\right), \qquad \text{by \eqref{eq:H3}} \\
%&\leq C\left( \eps_0^\alpha + \gamma \max_{u\in \cM(\Omega_0)} \Vert u \Vert^\alpha\right) , \qquad \text{by \eqref{eq:H1}}
%\label{eq:last-ineq}
%\end{align}
%The result follows by inserting \eqref{eq:last-ineq} into \eqref{eq:bound-proj-err}, and then inserting \eqref{eq:bound-proj-err} into \eqref{eq:err-pbdw-proof}.
\end{proof}
Theorem \ref{thm:bound1} shows that several ingredients are required in order to obtain a good quality reconstruction in $\Omega_1$ from a template geometry $\Omega_0$:
\begin{itemize}
\item The quality of the reduced basis $V_n(\Omega_0)$ in $\Omega_0$ must be high so that $\eps_0$ is small enough.
\item The transported manifold $\tau_{0\to1}\(\cM(\Omega_0)\)$ must be close the target manifold $\cM(\Omega_1)$ is the sense that $\eta$ is small enough.
\item The transported space $\widehat V_n(\Omega_1)=\widehat \tau_{0\to 1}(V_n(\Omega_0))$ must have ``a good alignment'' with the observation space $W\new{_m}$ in the sense that the stability constant $\beta(\widehat V_n(\Omega_1), W\new{_{m}})$ is bounded away from 0.
\item Finally, the transport of the space $V_n(\Omega_0)$ with $\widehat \tau_{0\to1}$ must approximate as well as possible the one with $\tau_{0\to1}$ so that $\gamma$ is small. 
\end{itemize}

\subsection{An alternative error bound based on subspace distances}
The reconstruction error bound \eqref{eq:reconstruction-error-1} given in Theorem \ref{thm:bound1} involves very natural quantities such as the Hausdorff distance between the target manifold $\cM(\Omega_1)$ and the transported one $\tau_{0\to1}\new{\(\cM(\Omega_0)\)}$. The bound \eqref{eq:reconstruction-error-1} may however be pessimistic in the sense that if $d_H( \tau_{0\to1} \(\cM(\Omega_0)\), \cM(\Omega_1))$ is large, then the bound will not guarantee a high quality (because $\eta$ is large). In this scenario, the reconstruction may however still be of decent quality if the transported subspace $\widehat \tau_{0\to1} (V_n(\Omega_0))$ does not deviate much compared to good quality reduced subspaces $V_n(\Omega_1)$ that one could compute in the target manifold $\cM(\Omega_1)$.

Theorem \ref{thm:bound2} quantifies this argument. It is a perturbative result that expresses to what extent the reconstruction is degraded between working directly with a linear subspace $V_n(\Omega_1)$ and a transported subspace $\widehat V_n(\Omega_1) = \widehat \tau_{0\to1} (V_n(\Omega_0))$. The result involves the Hausdorff distance between the unit spheres of these two spaces, which we denote by $\bS(V_n(\Omega_1))$ and $\bS(\widehat V_n(\Omega_1))$. The square of this distance can be written as
\begin{equation}
\begin{aligned}
d^2_H(\bS(\widehat V_n(\Omega_1)), \bS(V_n(\Omega_1)))
&= \max \left( 
\max_{\hat v\in \widehat V_n(\Omega_1) }  \frac{\Vert \hat v - P_{V_n(\Omega_1)} \hat v \Vert^2}{\Vert \hat v\Vert^2} ;  
\max_{v\in V_n(\Omega_1) }  \frac{\Vert v - P_{\widehat{V}_n(\Omega_1)} v \Vert^2}{\Vert v\Vert^2}
\right) \\
&= \max \left( 
1 - \beta^2(\widehat V_n, V_n);\,
1 - \beta^2(V_n, \widehat V_n)
\right) \\
&= 
1 - \min\left( \beta^2(\widehat V_n, V_n); \beta^2(V_n, \widehat V_n) \right).
\label{eq:hausdorff_sphere}
\end{aligned}
\end{equation}
\new{where the Pythagorean identity $\Vert v \Vert^2 = \Vert P_{V_n(\Omega_1)} v \Vert^2 + \Vert v - P_{V_n(\Omega_1)} v \Vert^2$ ($\forall v \in V_n(\Omega_1)$) has been used}.

\begin{theorem}
\label{thm:bound2}
Let $V_n(\Omega_1)$ be a linear subspace such that
\begin{align}
\max_{u \in \cM(\Omega_1)} \Vert u - P_{V_n(\Omega_1)} u \Vert &\leq \eps, \label{eq:hyp-b}\\
 \beta(V_n(\Omega_1), W) &\geq \underline{\beta} > 0. \label{eq:hyp-beta}
\end{align}
Let \new{$\widehat V_n(\Omega_1) = \widehat \tau_{0\to1} (V_n(\Omega_0))$} be a transported subspace from $\Omega_0$ to $\Omega_1$ such that
\begin{equation}
\label{eq:hypoth-dh}
d_H(\bS(\widehat V_n(\Omega_1)), \bS(V_n(\Omega_1))) \leq \delta_H.
\end{equation}
Then the reconstruction of $\cM(\Omega_1)$ with PBDW using $V_n(\Omega_1)$ is well-posed and the error is bounded by
$$
\max_{u\in \cM(\Omega_1)} \Vert u - A_{V_n(\Omega_1)} (P_W u) \Vert \leq \frac{\eps}{\underline{\beta}}.
$$
If we use $\widehat V_n(\Omega_1)$, the reconstruction is well posed if and only if
\begin{equation}
\delta_H < \underline{\beta},
\label{eq:delta-bound}
\end{equation}
and the reconstruction error is bounded by
\begin{align}
\label{eq:thm-2-main-result}
&\max_{u\in \cM(\Omega_1)} \Vert u - A_{\widehat V_n(\Omega_1)} (P_W u) \Vert \leq \frac{ \eps + 2\delta_H \max_{u\in \cM(\Omega_1)} \Vert P_{V_n+\widehat V_n} u \Vert }{\underline{\beta} (1-\delta_H/\underline{\beta})^{1/2}((2+\delta_H)/\underline{\beta}-1)^{1/2}}.
\end{align}
\end{theorem}

\begin{proof}
Let $u \in \cM(\Omega_1)$. By direct application of \eqref{eq:err-pbdw-Vhat}, we have
\begin{equation}
\Vert u - A_{\widehat V_n(\Omega_1)} (P_W u) \Vert \leq \frac{1}{\beta(\widehat V_n(\Omega_1), W)} \Vert u - P_{\widehat V_n(\Omega_1)} u\Vert.
\label{eq:err-pbdw-proof-bis}
\end{equation}
By the triangle inequality and hypothesis \eqref{eq:hyp-b} and \eqref{eq:hypoth-dh},
\begin{equation}
\label{eq:thm-2-tmp-0}
\Vert u - P_{\widehat V_n(\Omega_1)} u\Vert \leq \Vert u - P_{V_n(\Omega_1)} u\Vert + \Vert P_{V_n(\Omega_1)} u - P_{\widehat V_n(\Omega_1)} u\Vert \leq \eps + 2\delta_H \Vert P_{V_n+\widehat V_n} u \Vert .
\end{equation}
We next prove that
\begin{equation}
\label{eq:ineq-beta-thm-2}
\beta^2(\widehat V_n, W) \geq 1- (1-\underline{\beta}+\delta_H)^2 = \underline{\beta}^2 (1-\delta_H/\underline{\beta})((2+\delta_H)/\underline{\beta}-1).
\end{equation}
Note that this automatically guarantees that the reconstruction using $\widehat V_n(\Omega_1)$ is well-posed since, by hypothesis \eqref{eq:delta-bound}, we have $\delta_H<\underline{\beta}$ and therefore $\beta(\widehat V_n, W) >0$.

To prove \eqref{eq:ineq-beta-thm-2}, we start from the fact that
$$
\beta^2(\widehat V_n, W) = 1 - \max_{\hat v \in \widehat V_n} \frac{\Vert \hat v - P_W \hat v \Vert^2}{\Vert \hat v \Vert^2},
$$
and, by Jensen's inequality, we have that for any $\zeta>0$ 
\begin{equation}
\label{eq:tmp-thm-2}
\frac{\Vert \hat v - P_W \hat v \Vert^2}{\Vert \hat v \Vert^2}
\leq
(1+\zeta) \frac{ \Vert \hat v - v - P_W(\hat v - v)\Vert^2 }{ \Vert \hat v \Vert^2 }
+
(1+\zeta^{-1}) \frac{\Vert v - P_W v \Vert^2}{ \Vert \hat v \Vert^2 },\quad \forall v \in V_n .
\end{equation}
Now, on the one hand,
\begin{equation}
\label{eq:aux-1-thm-2}
\frac{ \Vert \hat v - v - P_W(\hat v - v)\Vert^2 }{ \Vert \hat v \Vert^2 } \leq \frac{ \Vert \hat v - v \Vert^2 }{ \Vert \hat v \Vert^2 }.
\end{equation}
On the other hand,
\begin{align}
\label{eq:aux-2-thm-2}
 \frac{\Vert v - P_W v \Vert^2}{ \Vert \hat v \Vert^2 }
&\leq  \frac{\Vert v \Vert^2}{\Vert \hat v \Vert^2} \max_{v\in V_n}   \frac{\Vert v - P_W v \Vert^2}{ \Vert v \Vert^2 }
\leq \frac{\Vert v \Vert^2}{\Vert \hat v \Vert^2}  (1-\underline{\beta}^2),\quad \forall v \in V_n,
\end{align}
where we have used \eqref{eq:hyp-beta} to derive the last inequality. 
Thus inserting bounds \eqref{eq:aux-1-thm-2} and \eqref{eq:aux-2-thm-2} into \eqref{eq:tmp-thm-2}, and setting $v=P_{V_n} \hat v$, we derive
\begin{align}
\frac{\Vert \hat v - P_W \hat v \Vert^2}{\Vert \hat v \Vert^2}
&\leq
(1+\zeta) \frac{\Vert \hat v - P_{V_n} \hat v \Vert^2}{\Vert \hat v \Vert^2}
+
(1+\zeta^{-1}) (1-\underline{\beta}^2) \frac{\Vert P_{V_n} \hat v \Vert}{\Vert \hat v \Vert} \\
&\leq (1+\zeta)  \max_{\hat v \in \widehat V_n}\frac{\Vert \hat v - P_{V_n} \hat v \Vert^2}{\Vert \hat v \Vert^2}
+ (1+\zeta^{-1}) (1-\underline{\beta}^2) \\
&\leq (1+\zeta) \delta_H^2
+ (1+\zeta^{-1}) (1-\underline{\beta}^2) ,\quad \forall \hat v \in \widehat V_n, \; \forall \zeta>0.
\end{align}
We can maximize the left-hand side over $\hat v \in \widehat V_n$ and minimize the right-hand side over $\zeta >0$. This yields
\begin{align}
1 - \beta^2(\widehat V_n, W)
&= \max_{\hat v \in \widehat V_n} \frac{\Vert \hat v - P_W \hat v \Vert^2}{\Vert \hat v \Vert^2} \\
&\leq \min_{\zeta>0} (1+\zeta) \delta_H^2 + (1+\zeta^{-1}) (1-\underline{\beta}^2) \\
&= (1-\underline{\beta}+\delta_H)^2,
\end{align}
which is the proof to inequality \eqref{eq:ineq-beta-thm-2}. We derive the final result \eqref{eq:thm-2-main-result} by inserting bounds \eqref{eq:ineq-beta-thm-2} and \eqref{eq:thm-2-tmp-0} into \eqref{eq:err-pbdw-proof}.
\end{proof}
From the error bound \eqref{eq:thm-2-main-result} from Theorem \ref{thm:bound2}, it follows that if the transported subspace $\widehat{V}_n(\Omega_1)$ deviates from $V_n(\Omega_1)$ by a quantity of the order $\delta_H \leq \eps / \max_{u\in \cM(\Omega_1)} \Vert u \Vert$, then
$$
\max_{u\in \cM(\Omega_1)} \Vert u - A_{\widehat V_n(\Omega_1)} (P_W u) \Vert \leq C \frac{\eps}{\underline{\beta}},
$$
for a relatively moderate constant $C\geq 1$. In this scenario, the reconstruction with the transported subspace is of the same quality as the one with the linear subspace $V_n(\Omega_1)$ (which we are avoiding to compute in order to speed-up the state estimation procedure).

\section{Transport routine $\tau$ and the routine \texttt{Best-Template}}
\label{sec:practical-realization}

\subsection{Computation of $\tau_{\Omega\to\Omega'}$ and $\widehat \tau_{\Omega\to\Omega'}$}
\label{sec:transport-routine}
We next describe a practical way of mapping snapshots and subspaces from a given geometry $\Omega_0$ to a target geometry $\Omega_1$. Our approach is based on building a one-to-one mapping between the two volumes $\Omega_0$ and $\Omega_1$. It involves three steps:
\begin{enumerate}
\item \emph{Surface matching:} The task is to compute a map between $\partial \Omega_0$ and $\partial \Omega_1$. For this, we use the so-called Large Deformation Diffeomorphic Metric Mapping (LDDMM, see for instance \cite{lddmm}) method. In practice, the output of this method is an invertible and smooth mapping $T_{(\text{LDDMM})}: \partial \Omega_0\to \partial \Omega_1'$ between $\partial \Omega_0$ and an intermediate surface $\partial \Omega_1'$ which is close to the target surface $\partial \Omega_1$. The mapping is such that, if $\partial\Omega_0=\partial\Omega_1$, then $T_{(\text{LDDMM})}(x) = x,\,\forall x \in \partial\Omega_0$. The surface misfit between $\partial\Omega'_1$ and $\partial\Omega_1$ is corrected in step 3 with an interpolation post-processing.
\item \emph{Extrapolation of the surface map to the entire volume:} We make a harmonic extension on $\Omega_0$ and we find a displacement field $d_0\in H^1(\Omega_0)^d$ such that
\begin{align}
\Delta d_0 &= 0, \quad \text{ in }\Omega_0\\
d_0(x) &= T_{(\text{LDDMM})}(x)-x, \quad \forall x \in \partial \Omega_0.
\label{eq:lddmm_extension}
\end{align}
Note that $d_0 = 0$ if $\Omega_0=\Omega_1$. We define the volumetric mapping
\begin{align}
T_{0\to 1'} :\Omega_0 &\to \Omega_1'  \\
x_0 & \mapsto x_1 = T_{0\to 1'}(x_0)\coloneqq x_0 + d_0(x_0).
\end{align}
This map is invertible and $T^{-1}_{0\to 1'} = T_{1'\to 0}$. We further define the functional mapping
\begin{align}
\phi_{0\to 1'} : V(\Omega_0) &\to V(\Omega_1')  \\
f & \mapsto \phi_{0\to 1'} (f)(x'_1) \coloneqq f \circ T_{1'\to0}(x'_1),\quad \forall x'_1\in \Omega_1'.
\end{align}
\item \emph{Interpolation:} Since in general $\Omega_1' \neq \Omega_1$, we add an interpolation operator  $\cI_{1'\to 1}: V(\Omega_1') \mapsto V(\Omega_1)$, so that the final mapping is
\begin{align}
\tau_{0\to 1} : V(\Omega_0) &\to V(\Omega_1)  \\
f & \mapsto \tau_{0\to 1} (f) \coloneqq \cI_{1' \to 1} \( \phi_{0\to 1'}(f) \).
\end{align}
\end{enumerate}
Note that the map $\tau_{0\to 1}$ may not exist if the spaces $V(\Omega_0)$ and $V(\Omega_1)$ are chosen of very different nature (very different regularity) or if certain physical quantities need to be preserved. One relevant example for fluid and biomedical applications is the space of divergence free fields where $V(\Omega_0)=H(\text{div}, \Omega_0)$ and $V(\Omega_1)=H(\text{div}, \Omega_1)$. In this case, for any $f\in H(\text{div}, \Omega_0)$, we have $\tau_{0\to 1}f \in H^1(\Omega_1)$ but the function may not be divergence free. One remedy in this case is to add a post-process with the Piola transform. We therefore update the abstract definition of $\tau_{0\to 1}$ by adding a post-process mapping $p$ to allow this type of scenario,
\begin{align}
\tau_{0\to 1} : V(\Omega_0) &\to V(\Omega_1)  \\
f & \mapsto \tau_{0\to 1} (f) \coloneqq p \circ  \cI_{1' \to 1} \( \phi_{0\to 1'}(f) \).
\label{eq:post_process_map}
\end{align}

In our reconstruction method, we need to transport subspaces $V_n(\Omega_0) \subseteq V(\Omega_0)$ to subspaces of $V(\Omega_1)$. Note that in general the image of $V_n(\Omega_0) $ by $\tau_{0\to1}$, defined as
$$
\tau_{0\to 1}(V_n(\Omega_0)) \coloneqq \{ \tau_{0\to 1}(v) \in V(\Omega_1) \cond v\in V(\Omega_0) \},
$$
is not a linear subspace of  $V(\Omega_1)$ unless $\tau_{0\to1}$ is a linear map. Due to this, given that in our approach we need to map subspaces into subspaces, we choose to define the image of $V_n(\Omega_0)$ with respect to a given basis $\cB=\vspan\{\varphi_1,\dots, \varphi_n\}$ of $V_n(\Omega_0)$ as
\begin{align}
\widehat \tau_{0\to1} (\new{V_n(\Omega_0)}, \cB) \coloneqq \vspan\{ \tau_{0\to 1}(\varphi_1),\dots, \tau_{0\to 1}(\varphi_n)  \}.
\end{align}

\subsection{The  \texttt{Best-Template} routine \BT}
\label{sec:best-template}
The goal of this routine is to identify for each new target geometry $\Omega \in \rG$ which template geometry $\Omega_t \in \rG_\templates $ has the most appropriate linear subspace $V_n(\Omega_t)$ that we have to transport to $\Omega$.

Given a target geometry $\Omega \in \rG$ and a template geometry $\Omega_t \in \rG_\templates$, the reconstruction error is bounded by (see \eqref{eq:err-pbdw-Vhat})
\begin{equation}
\max_{u \in \cM(\Omega)}\Vert u - A_{\widehat\tau_{\Omega_t \to \Omega}( V_n(\Omega_t))}(\new{P_{W_m}} u) \Vert \leq \frac{1}{\beta(\widehat\tau_{\Omega_t \to \Omega} (V_n(\Omega_t)),\new{W_m})} \delta^{(\wc)}_{\Omega_t \to \Omega},
\label{eq:err-pbdw-Vhat-bis}
\end{equation}
where
\begin{equation}
\label{eq:err-fwd-wc}
\delta^{(\wc)}_{\Omega_t \to \Omega} \coloneqq \max_{u \in \cM(\Omega) } \Vert u - P_{ \widehat \tau_{\Omega_t \to \Omega}(V_n(\Omega_t))} u \Vert ,\quad \forall\; \Omega_t \in \rG_\templates .
\end{equation}
Alternatively, if we study errors in the average sense,
\begin{equation}
\bE( \Vert u - A_{\widehat\tau_{\Omega_t \to \Omega} (V_n(\Omega_t))}(\new{P_{W_m}} u) \Vert^2 ) \leq \frac{1}{\beta^2(\widehat\tau_{\Omega_t \to \Omega} (V_n(\Omega_t)), \new{W_m})} \(\delta^{(\ms)}_{\Omega_t \to \Omega}\)^2,
\label{eq:err-pbdw-Vhat-ms}
\end{equation}
with
\begin{equation}
\label{eq:err-fwd}
\delta^{(\ms)}_{\Omega_t \to \Omega} \coloneqq \bE( \Vert u - P_{ \widehat \tau_{\Omega_t \to \Omega}(V_n(\Omega_t))} u \Vert^2 )^{1/2},\quad \forall\; \Omega_t \in \rG_\templates .
\end{equation}
Ideally, we would like to find the template $\Omega_t$ that miminizes the upper bound \eqref{eq:err-pbdw-Vhat-bis} or \eqref{eq:err-pbdw-Vhat-ms}, that is, find
$$
\Omega_t^* \in  \argmin_{\Omega_t \in \rG_\templates} \frac{1}{\beta(\widehat\tau_{\Omega_t \to \Omega} (V_n(\Omega_t)), \Wm(\Omega))} \delta^{(\star\star)}_{\Omega_t \to \Omega},
$$
where $(\star\star)$ means $(\wc)$ or $(\ms)$ depending on the desired setting to study the errors. \om{Note that this criterion depends on the observation space $\Wm(\Omega)$ that we use for the reconstruction in $\Omega$, and there are two scenarios:
\begin{itemize}
\item We can use directly this criterion if we consider that $\Wm(\Omega)$ is known with enough advance, and that we have enough time to compute $\frac{1}{\beta(\widehat\tau_{\Omega_t \to \Omega} (V_n(\Omega_t)), \Wm(\Omega))} \delta^{(\star\star)}_{\Omega_t \to \Omega}$ for all $\Omega_t \in \rG_\templates$. Note however that this is a very costly operation in general.
\item There are settings in which a suboptimal criterion that does not involve $\Wm(\Omega)$ but that is computationally faster might be required. One scenario in which this is the case is when one whishes to study several different observation spaces $\Wm(\Omega)$. Another scenario concerns  applications in which one cannot assume that $\Wm(\Omega)$ is known with enough advance. In such cases, the selection of the template geometry has to be performed in a reduced computational time in the online phase. This is the case of numerous biomedical problems which we are particularly targeting in our numerical experiments.
\end{itemize}
In the following, we present a strategy for the second, more challenging scenario. Our approach is based on Theorem \ref{thm:bound2}. From bound \eqref{eq:thm-2-main-result} of that theorem,} it follows that a strategy to find the best template is to minimize over the Hausdorff distance 
\begin{equation}
d_H(\bS(\widehat \tau_{\Omega_t \to \Omega} (V_n(\Omega_t))), \bS(V_n(\Omega))),
\label{eq:choice-rho-v1}
\end{equation}
between a good linear subspace $V_n(\Omega)$ (coming, for example, from forward reduced modeling) and the transported subspace $\widehat \tau_{\Omega_t \to \Omega} \(V_n(\Omega_t)\)$. With this strategy, the output to select the best-template routine is thus
\begin{equation}
\label{eq:BT-running} 
\BT(\Omega) \in \argmin_{\Omega_t \in \rG_\templates}  d_H(\bS(\widehat \tau_{\Omega_t \to \Omega} \(V_n(\Omega_t) \)), \bS(V_n(\Omega))).
\end{equation}
In order to perform this selection in real time, we need to estimate quickly the map
$$
\Omega \in \rG \to \left\lbrace d_H(\bS(\widehat \tau_{\Omega_t \to \Omega}\( V_n(\Omega_t)\)), \bS(V_n(\Omega))) \cond \Omega_t\in \rG_\templates\right\rbrace.
$$
In our work, this is performed with a Multidimensional Scaling approach (MDS, see e.g.~\cite{MN1995, TDL2000, DG2005, GGKC2020}). We next describe the main steps.

\begin{remark}
%  \textcolor{gray}{Note that another criterion that does not involve $\Wm(\Omega)$ is to work with $\delta^{(\star\star)}_{\Omega_t \to \Omega}$. This strategy was studied in our numerical tests but it was outperformed by the criterion \eqref{eq:choice-rho-v1} discussed in the main text. We conjecture that the reason for this is  related to the fact that $\delta^{(\star\star)}_{\Omega_t \to \Omega}$ is connected to the approximation quality of the forward linear subspaceing problem instead of our current inverse reconstruction problem.}
  \om{Working with the quantity \eqref{eq:choice-rho-v1} is theoretically justified by Theorem \ref{thm:bound2}. Note that the theorem requires enough stability in the sense that \eqref{eq:hyp-beta} needs to be satisfied. This is taken as an assumption in the following development. In our numerical tests, this conditions is satisfied thanks to the quality of the data for the applications we focus on.}
\end{remark}

%\newpage
%A first option is to minimize only over the component $d_{\Omega_t \to \Omega}$. Another option that
% Another option, which has produced better numerical results, is to minimize over the Hausdorf distance 
%$$
%d^2_H(\bS(\widehat V_n(\Omega_1)), \bS(V_n(\Omega_1))
%$$
%
%
%
%This quantity is related to the approximation quality in the forward linear subspaceing problem instead of the inverse state estimation problem.  Our desired output for the best-template routine is thus
%\begin{equation}
%\label{eq:BT-running}
%\BT(\Omega) \in \argmin_{\Omega_t \in \rG_\templates} \delta^{(\star\star)}_{\Omega_t \to \Omega},
%\end{equation}
%In fact, we will see in the discussion that follows that this quantity will have to be symmetrized for technical reasons but let us keep \eqref{eq:BT-running} as the running definition for $\BT(\Omega)$.
%
%In order to perform this selection in real time, we need to estimate quickly the map
%$$
%\Omega \in \rG \to \{ \delta^{(\star\star)}_{\Omega_t \to \Omega} \cond \Omega_t\in \rG_\templates\}.
%$$
%In our work, this is performed with a Multidimensional Scaling approach (MDS, see e.g.~\cite{MN1995, TDL2000, DG2005, GGKC2020}). We next describe the main steps.

\paragraph{Step 1: Voxelize geometries:} To ease the manipulation and comparison between different domains, we work with voxelized descriptions of them involving a uniform grid mesh of $N_\vox$  cells. Therefore, instead of working with a given domain $\Omega\subset \bR^d$, we will actually manipulate vectors $v_{\Omega}\in \bR^{N_\vox}$ such that for all $i=1,\dots, N_\vox$, the voxel entry $v^{(i)}_{\Omega}$ is equal to the volume portion of the associated cell $i$ of the mesh. Ideally, the size of the grid mesh $N_\vox$ should be large enough in order to guarantee an isomorphism between the domains $\Omega \in \rG$ and their corresponding voxelizations $v_\Omega$.

The family of geometries $\rG$ is therefore replaced in practice by the voxelized representation,
$$
\rG \sim \rV \coloneqq \{ v_\Omega \in \bR^{N_\vox} \cond \Omega \in \rG \}.
$$ 
Similarly,
$$
\rG_\templates \sim \rV_\templates \coloneqq \{ v_\Omega \in \bR^{N_\vox} \cond \Omega \in \rG_\templates \}.
$$
As a result of the voxelization, we will alternatively write the manifold set of solutions $\cM(\Omega)$ as $\cM(v_\Omega)$ for all $\Omega\in \rG$. Also, in practice we will construct a best template mapping of the form
$$
\BT: \bR^{N_v} \in \rV \to \rV_\templates .
$$
\paragraph{Learning Phase -- Step 1: MDS:} 
We consider the manifold set
$$
\cS \coloneqq \{ \cM(v_\Omega) \cond v_\Omega \in \rV \}.
$$
Our goal is to find a low dimensional representation of $\cS$ using our database of $K$ templates,
$$
\cS_\templates \coloneqq \{ \cM(v_\Omega) \cond v_\Omega \in \rV_\templates \}.
$$
For this, suppose that $\cS$ is equipped with a metric $\rho$. The exact choice for $\rho$ will be specified later on. We then assemble the matrix of pairwise square distances between elements of $\cS_\templates$,
\begin{equation}
\pbD = (d_{i,j})_{1\leq i, j \leq K}, \qquad d_{i,j} =  \rho^2\left(\cM(v_{\Omega_i}), \cM(v_{\Omega_j}) \right).
\label{eq:dij}
\end{equation}
The vanilla version of MDS seeks to find vectors $x_1,\dots, x_K$ from an Euclidean space $\bR^p$ of small dimension $p$ such that
$$
\Vert x_i - x_j \Vert_{\ell^2(\bR^p)}^2 = d_{i,j}, \quad 1\leq i, j \leq K.
$$
The solution to this problem, if it exists, is not unique because if \new{ $\pbX^*=(x^*_1 \vert \dots\vert x^*_K) \in \bR^{p\times K}$} is a solution, then \new{$\pbX^*_c = (x^*_1+c\,\vert \dots\vert\, x^*_K+c)$} is also a solution for any vector $c\in \bR^p$. We therefore add a constraint in which we search for the unique centered solution such that $\sum_{i=1}^p x^*_{i,j} =0$ for all $j=1,\dots, K$. One can easily prove that, if such a centered solution $\pbX^*$ exists, then it satisfies the equation
\begin{equation}
(\pbX^*)^T \pbX^* = \pbC,
\label{eq:ip-emb}
\end{equation}
with
$$
\pbC \coloneqq - \frac 1 2 \pbH \pbD \pbH, \qquad \pbH \coloneqq \pbI - \frac{1}{K} ee^T, 
\qquad 
e \coloneqq 
\underbrace{(1, \dots, 1)}_{K}\, ^T.
$$
The matrix $\pbC$ resembles a covariance matrix in that if the original pairwise distances represent Euclidean distances in a $p$-dimensional space, $\pbC$ will be symmetric and positive semidefinite of rank $p$. Since $\pbC$ is symmetric, its eigenvalue decomposition is of the form
$$
\pbC = \pbV \mathbf{\Lambda} \pbV^T,
$$
where $\pbV=(v_1\vert \dots\vert v_K) \in \bR^{K\times K}$ is a unitary matrix and $\mathbf{\Lambda} = \diag(\lambda_1, \dots, \lambda_K) $ is a diagonal matrix containing the eigenvalues in the diagonal. We sort them in decreasing order $\lambda_1\geq \dots \geq \lambda_K$.

If $\pbC$ is positive definite of rank $p$, we have $\lambda_1\geq \dots\geq \lambda_p>0$ and $\lambda_i=0$ for $p<i\leq K$. In this case, we can exactly represent the objects as points in a $p$-dimensional space, in such a way that the square of the Euclidean distance $\Vert x_i - x_j \Vert^2_{\ell^2(\bR^p)}$ between each pair of points is exactly equal to $d_{i,j}= \rho^2\left(\cM(v_{\Omega_i}), \cM(v_{\Omega_j}) \right)$. To find the points, we consider the eigenvectors $v_1,\dots, v_p$ associated to the nonnegative eigenvalues and assemble the matrices
$$
\pbV_p = (v_1\vert \dots\vert v_p) \in \bR^{K\times p}, \quad \mathbf{\Lambda}_p \in \bR^{K\times p}, 
$$
\new{where $\(\Lambda_p\)_{ii} = \lambda_i$ and $\( \Lambda_p \)_{ij} = 0$, for $i\neq j$. We then set, for $r \leq p$ ($r \in \bN^*$):}
$$
\pbX = \mathbf{\Lambda}_r^{1/2} \pbV^T_r.
$$
Of course, in general $\pbC$ need not be positive semi-definite, which will not be true if there is no $p$-dimensional embedding representing the $K$ objects with specified pair-wise distances $d_{i,j}$. In such cases, the standard MDS procedure is to embed the data using only the positive eigenvalues. This yields an approximate embedding, whose quality depends on the importance of the eigenmodes that are discarded.

The selection of the metric for the manifold $\cS$ plays a critical role in the ability of MDS to find a low dimensional representation of $\cS$.  
%The celebrated ISOMAP procedure \cite{TDL2000, DG2005} works with geodesic distances $\rho$ between objects from $\cS$. In our case, we follow a different avenue \blue{[It is ISOMAP so it is not a different avenue.]} since we want to incorporate the quality of approximation with transported linear subspaces in the notion of distance. Therefore,
Ideally we would like to use $d^2_H(\bS(\widehat \tau_{\Omega' \to \Omega}\( V_n(\Omega')\)), \bS(V_n(\Omega)))$ as defined in \eqref{eq:choice-rho-v1} but the main obstacle is that this quantity is not symmetric. This is the reason why we use the symetrized version
\begin{align}
\rho^2( \cM(v_\Omega), \cM(v_{\Omega'}) )
&\coloneqq \frac 1 2 d^2_H(\bS(\widehat \tau_{\Omega' \to \Omega} \(V_n(\Omega')\)), \bS(V_n(\Omega))) + \frac 1 2 d^2_H(\bS(V_n(\Omega')), \bS(\widehat \tau_{\Omega \to \Omega'} \(V_n(\Omega)\))).
\label{eq:symmetrized}
\end{align}
Note that the above mapping $\rho:\cS\times\cS\mapsto \bR_+$ does not define a distance in the classical sense because it does not satisfy the triangle inequality. Despite this, the fact that it is symmetric is sufficient to perform the MDS procedure. We will see that this choice yields good results despite the fact that we do not work with a metric. The success of our choice may be connected to the fact that our function $\rho$ involves a notion of ordering since we have that $0=\rho( \cM(v_\Omega), \cM(v_\Omega) ) < \rho( \cM(v_\Omega), \cM(v_{\Omega'}) )$ if $\Omega'\neq \Omega$.

\paragraph{Learning Phase -- Step 2: Voxelization-to-embedding-space Mapping:} The final element in our procedure is to build a mapping between the voxelization $v_\Omega$ of a geometry $\Omega\in \rG$ and the corresponding point $x_\Omega\in \bR^p$ in the low dimensional parametrization of $\cS$. In our case, this step is done by a simple linear least-squares procedure but of course other options could be considered. We search for a minimizer of
$$
\min_{\pbW\in \bR^{N_\vox\times p}} \frac{1}{2K} \sum_{\Omega \in \rG_\templates} \Vert \pbW^T v_\Omega -x_\Omega \Vert^2_{\ell^2(\bR^p)}.
$$
Denoting $\pbV = (v_{\Omega_1}\vert \dots\vert v_{\Omega_K})\in \bR^{N_\vox \times K}$ and $\pbX = (x_{\Omega_1}\vert \dots\vert x_{\Omega_K}) \in \bR^{p\times K}$, the solution $\pbW$ with minimal norm satisfies the least-squares equation
$$
\pbV\pbV^T \pbW = \pbV \pbX,
$$
which can be solved by classical least-squares inversion techniques.

\paragraph{Practical Application of the routine \BT:} Once the above learning steps have been performed, given a domain $\Omega \in \rG$ we can quickly find the best template from $\rG_\templates$ by performing the following steps:
\begin{itemize}
\item Compute the corresponding voxelization $v_\Omega$ of the target geometry $\Omega$.
\item Find the representation of $\cM(v_\Omega)$ in the low-dimensional space by computing $x_\Omega = \pbW^T v_\Omega$.
\item Find the template geometry  which is the closest in the embedding
\begin{equation}
\Omega^*_t \in \argmin_{\Omega_t \in \rG_\templates} \Vert x_\Omega - x_{\Omega_t} \Vert^2_{\ell_2(\bR^p)},
\label{eq:practical-choice-omega-t}
\end{equation}
and set $\BT(\Omega) = \Omega^*_t$. This choice is justified from the following fact: our original minimization problem is \eqref{eq:BT-running}, that is, to find
\begin{equation}
\min_{\Omega_t \in \rG_\templates}  d_H(\bS(\widehat \tau_{\Omega_t \to \Omega} \(V_n(\Omega_t)\)), \bS(V_n(\Omega))).
\end{equation}
By definition \eqref{eq:symmetrized} of the metric $\rho$,
$$
d_H(\bS(\widehat \tau_{\Omega_t \to \Omega} (V_n(\Omega_t))), \bS(V_n(\Omega))) \leq \sqrt{2} \rho( \cM(v_\Omega), \cM(v_{\Omega'}) ) \approx \sqrt{2} \Vert x_\Omega - x_{\Omega'} \Vert^2_{\ell_2(\bR^p)}.
$$
Therefore, our choice \eqref{eq:practical-choice-omega-t} for $\Omega^*_t$ guarantees that
$$
\min_{\Omega_t \in \rG_\templates}  d_H(\bS(\widehat \tau_{\Omega_t \to \Omega} \( V_n(\Omega_t) \) ), \bS(V_n(\Omega))) \lesssim \sqrt{2} \Vert x_\Omega - x_{\Omega^*_t} \Vert^2_{\ell_2(\bR^p)}.
$$
\end{itemize}

\section{Numerical example}
\label{sec:numerical_example}

The proposed methodology is general and, among the many different applications that could be envisaged, problems from the field of biomedicine emerge as particularly relevant. As such, we next present a numerical example on this topic related to the task of reconstructing 3D blood velocity flows from Doppler ultrasound velocity images (see \cite{GGLM2021, GLM2021}). The tests are performed on synthetically generated observations due to our lack of real data. The linear observation functions $\{\ell_i\}_{i=1}^m$ will thus be defined in order to mimic the output of real ultrasound images.

Sections \ref{sec:tests-geom} to \ref{sec:tests-mds} give details on the test case, and outline the steps performed for the training phase. The training follows exactly the guidelines given in section \ref{sec:strategy}. Section \ref{sec:tests-perf} quantifies and illustrates the good performance of the reconstruction strategy.

\subsection{Geometry}
\label{sec:tests-geom}
In our example, the family $\rG$ of geometries is a set of 3D Venturi tubes with variations on three geometrical parameters concerning the tube coarctation (see Figure \ref{fig:venturi_sampling}). The parameters are the coarctation length $S_l$, its radius $S_r$, and its position along the $y-$axis $S_x$. The ranges of the geometrical parameters are $S_r \in [1.4,2.6]$ mm, $S_l \in [0.8L,1.2L]$ and $S_x \in [5,11]$ mm. The length of the tube is fixed to $L = 5 \text{ cm}$, and its diameter to $D = 0.4 \text{ cm}$.

\paragraph{Training Phase:} We work with $K=64$ template geometries for the database $\rG_\templates$. They are computed using a uniform grid sample on the three geometrical parameters. 
\begin{figure}[!htbp]
\centering
\includegraphics[height=4cm]{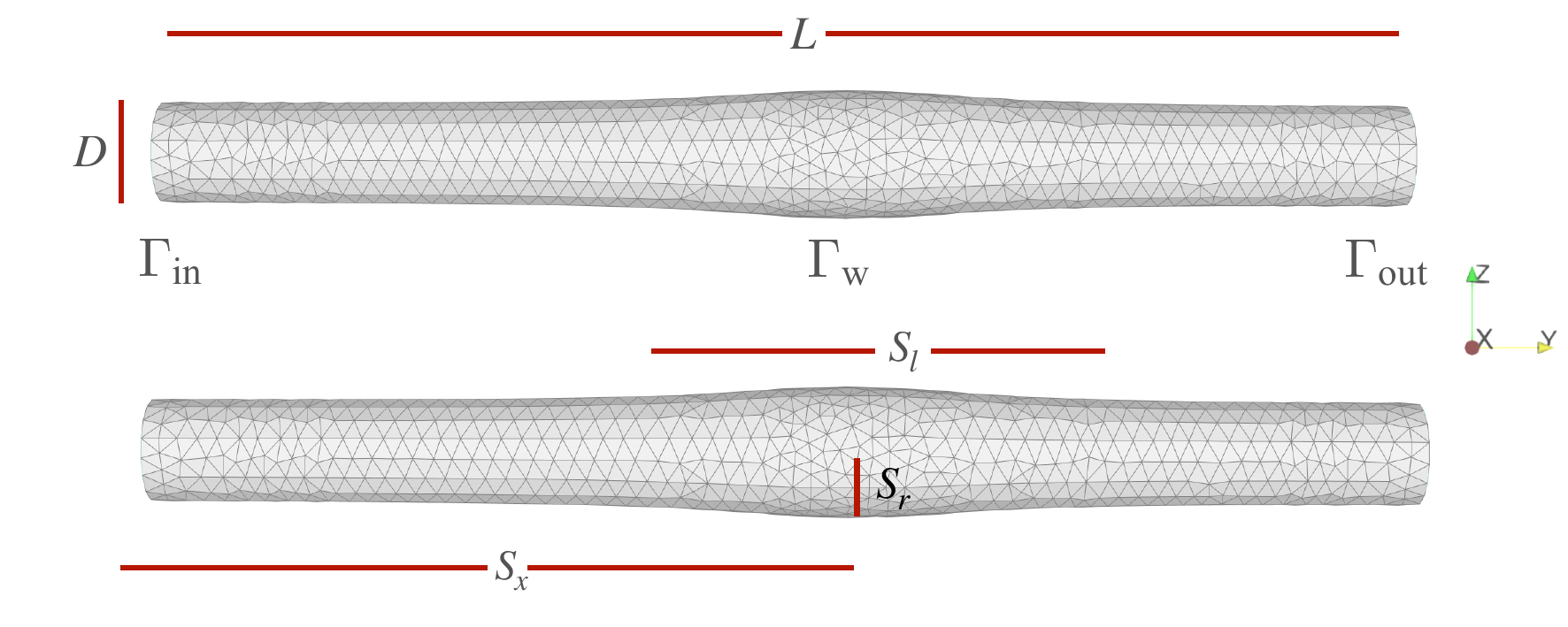}
\caption{Scheme for the generation of the set $G$.}
\label{fig:venturi_sampling}
\end{figure}

\subsection{Physics, solution manifold $\cM(\Omega)$, and linear subspace $V_n(\Omega)$}
\label{sec:tests-manifold}
We assume that the fluid is governed by the Stokes equations defined, for a given $ \Omega \in \rG$, as the problem of finding the velocity $u \in \[H^1\(\Omega;[0,T]\)\]^3$ and the pressure $p \in L^2(\Omega; [0,T])$ such that:
\begin{equation}
\left\lbrace
\begin{aligned}
\partial_t u - \mu \Delta u  + \nabla p = 0 & \text{ in } \Omega \\
\nabla \cdot u = 0 & \text{ in } \Omega \\ 
u = \(0,0,0\) & \text{ on } \Gamma_{\text{w}} \\
u = u_0 \( 0, 1 - \frac{x^2 + z^2}{(D/2)^2}, 0 \) \sin\(2 \pi t\) & \text{ on } \Gamma_{\text{in}} \\
\( \frac{\nabla^T u  + \nabla u}{2} - p\pbI \) \cdot n = (0,0,0) & \text{ on } \Gamma_{\text{out}}.
\end{aligned}
\right.
\label{eq:stokes}
\end{equation}
where $\pbI$ is an identity matrix of size three, $n$ is a unitary vector pointing outwards the working domain, and $u_0\in \bR_+$. The boundary $\partial \Omega$ is decomposed into 3 disjoint subdomains,
$$
\partial \Omega = \Gamma_{\text{in}} \cup \Gamma_{\text{out}} \cup \Gamma_{\text{w}},
$$
where $\Gamma_{\text{in}}$ is the inflow part, $\Gamma_{\text{out}}$ the outflow, and $\Gamma_{\text{w}}$ corresponds to the walls (see Figure \ref{fig:venturi_sampling}). 

In our example, we reconstruct velocities taking $V(\Omega)=[L^2(\Omega)]^3$ as the ambient reconstruction space. Note that this does not match with the space $[H^1(\Omega)]^3$ in which velocity is defined in the Stokes equation. This choice was made in order to target the reconstruction of the field and not its derivatives. 

For each $\Omega \in \rG$, we work with the manifold
$$
\cM(\Omega) \coloneqq \{ u(y) \cond y \in \rY\},
$$
with
$$
\rY \coloneqq \{ y = (t, u_{\text{0}}, \mu) \in [0,0.5 \text{ s.}] \times  [0.01, 1 \text{ cm/s}] \times [0.01, 0.1 \text{ P}] \}.
$$
\paragraph{Training Phase:} For each $\Omega \in G_\templates$, we compute a finite training subset of $\cM(\Omega)$ with $N_s=12~800$ snapshots, and we compute its Proper Orthogonal Decomposition (POD). The parameters to generate the snapshots are sampled from a uniform random distribution. Appendix \ref{app:stokes-solver} gives some details on the discretization and the solver used to generate them. The reduced order model $V_n(\Omega)$ is the subspace spanned by the POD eigenfunctions associated to the $n=20$ most energetic modes.

\subsubsection{Example of $\hat{\tau}_{0 \rightarrow 1}$ for mass conservative fields}

We have described in section \ref{sec:transport-routine} how fields are transported among domains. Let us illustrate the methodology with a numerical example between two geometries $\Omega_0$ and $\Omega_1$, as shown in figure \ref{fig:tau_example}.a and \ref{fig:tau_example}.b, respectively. Let $v_{st} \in \[H^1(\Omega)\]^3$ be a divergence free vector field, depicted on figure \ref{fig:tau_example}.a and solution to the Stokes problem \eqref{eq:stokes}, a snapshot in the training set of $\Omega_0$. In figure \ref{fig:tau_example}.b we observe the result of the shape registration via LDDMM (implemented using \cite{JMLR:v22:20-275}) computed from \eqref{eq:lddmm_extension}. Mass conservation is not preserved nonetheless. In order to convey a divergence free field in the arrival geometry we define the operator $p$ from equation \eqref{eq:post_process_map} as the Piola transform $p : [H^1(\Omega_0)]^3 \mapsto [H^1(\Omega_1)]^3$ (see \cite{ciarlet1988} or \cite{guibert2014}):
$$
p(v) = \frac{\( I_{3 \times 3} + \nabla \[ \cI_{1 \to 1'} \circ \phi_{0 \to 1'} \( d_0 \) \] \) }{ \text{det} \( I_{3 \times 3} + \nabla \[ \cI_{1' \to 1} \circ \phi_{0 \to 1'} (d_0) \] \) } \cI_{1' \to 1} \circ \phi_{0 \to 1'} \( v \).
$$

In figure \ref{fig:tau_example}.c we observe how this transformation recovers mass conservation in $\Omega_1$. The underlying mechanism of this operator is well illustrated with the scaling factor of figure \ref{fig:tau_example}.d.

\begin{figure}[!htbp]
\centering
\subfigure[Stokes snapshot $v_{st}$]{
\includegraphics[height=0.8cm]{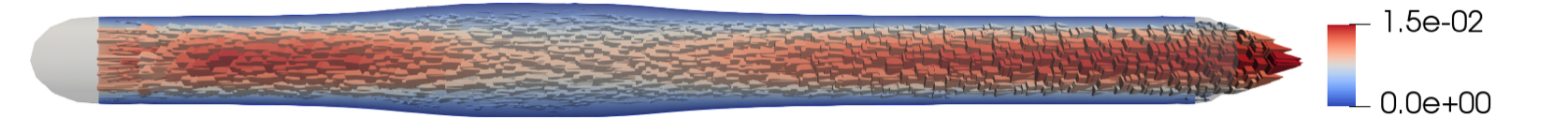}}
\subfigure[$\phi_{0\to 1'}(v_{st})$]{
\includegraphics[height=0.8cm]{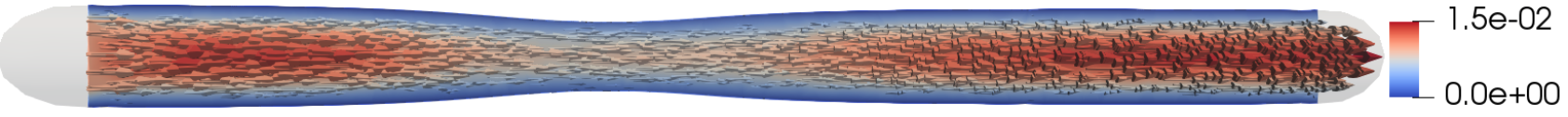}}
\subfigure[$p \circ \cI_{1' \to 1} \circ \phi_{0 \to 1'}(v_{st})$]{
\includegraphics[height=0.8cm]{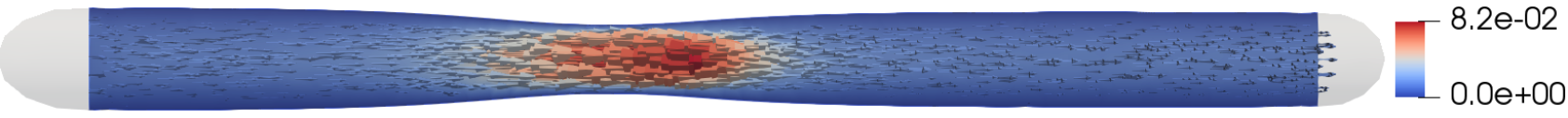}}
\subfigure[$ \text{det} \( I_{3 \times 3} + \nabla \[ \cI_{1' \to 1} \circ \phi_{0 \to 1'} (d_0) \] \) $]{
\includegraphics[height=0.8cm]{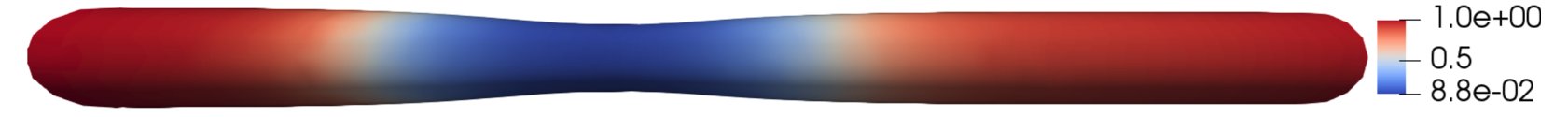}}
\caption{A divergence free field transported among geometries with a different value for $S_r$.}
\label{fig:tau_example}
\end{figure}

\subsection{MDS}
\label{sec:tests-mds}
We compute the MDS from the spectrum of the inner product matrix \eqref{eq:ip-emb}. To do so, we first compute the matrix $\pbD=(d_{i,j})_{1\leq i, j \leq K}$ of pairwise distances between the $K=64$ templates (see \eqref{eq:dij}). Each entry $d_{i,j}$ is computed using formula \eqref{eq:symmetrized} to quantify distances between two manifolds on different geometries.

To visually illustrate the methodology, we select a subset of $\widetilde K =16$ and show in Figure \ref{fig:hausdorff_example} the values $d_{i,j}$ of the matrix $\pbD$. Figure \ref{fig:points-low-dim} shows the positions $x_{\Omega}$ in the reduced Euclidean space of dimension $p=2$ for the $\widetilde K$  geometries. It is interesting to remark that the low dimensional representation of the geometries reflects the main differences in the geometrical parameters despite that the MDS methodology is fully non-parametric. The figure shows that the ``dominant'' parameter that drives metric changes is the radius $S_r$ since the points $x_\Omega$ tend to cluster following its values. For $\widetilde K =16$ geometries a bi-dimensional representation is enough to get a good embedding. For $K=64$ geometries, we work in $\bR^3$.
\begin{figure}[!htbp]
\centering
\subfigure[Pair-wise distances]{
\includegraphics[height=5cm]{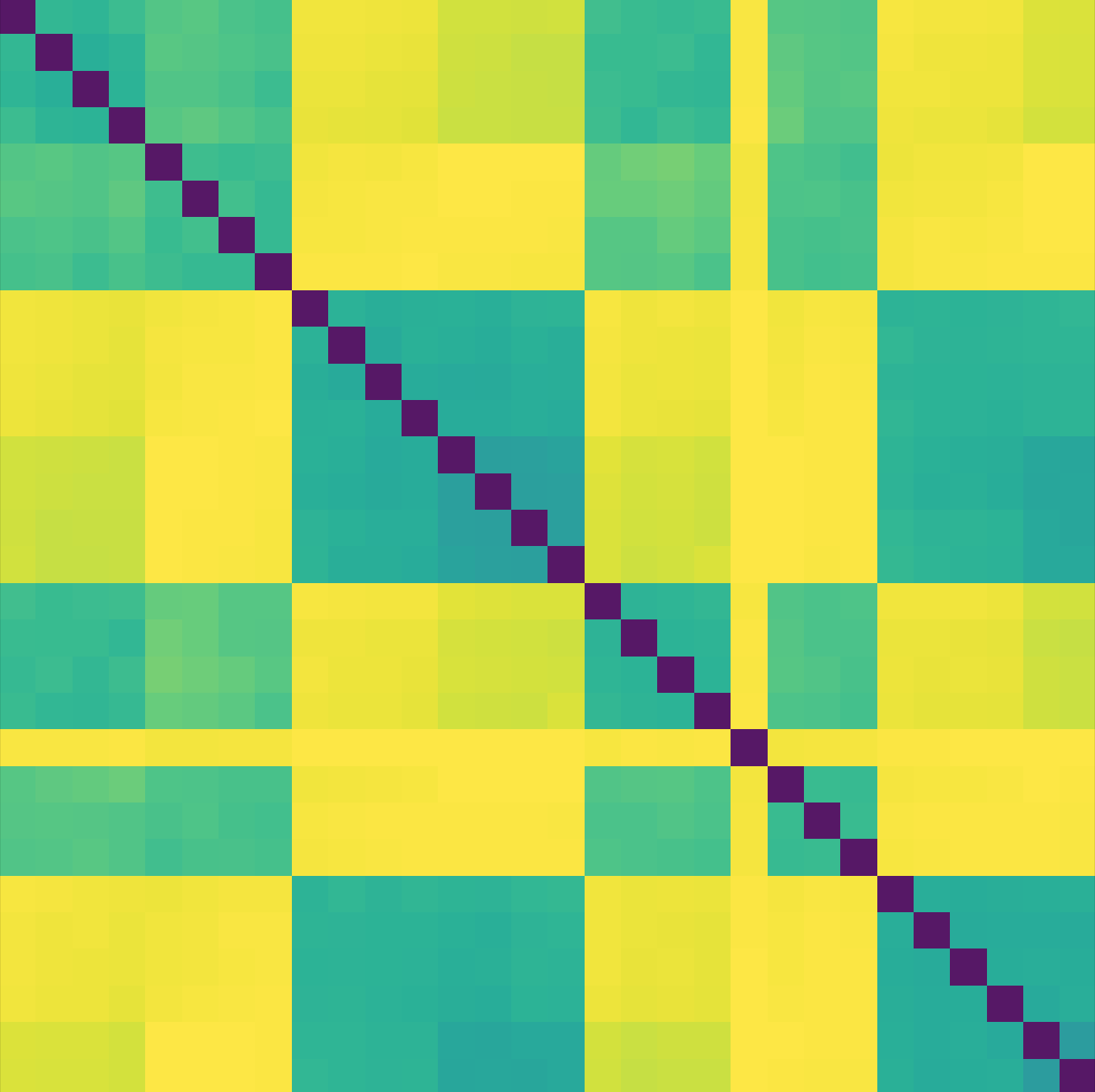}
\label{fig:hausdorff_example}
}
\subfigure[MDS coordinates of $X_\Omega$]{
\includegraphics[height=5cm]{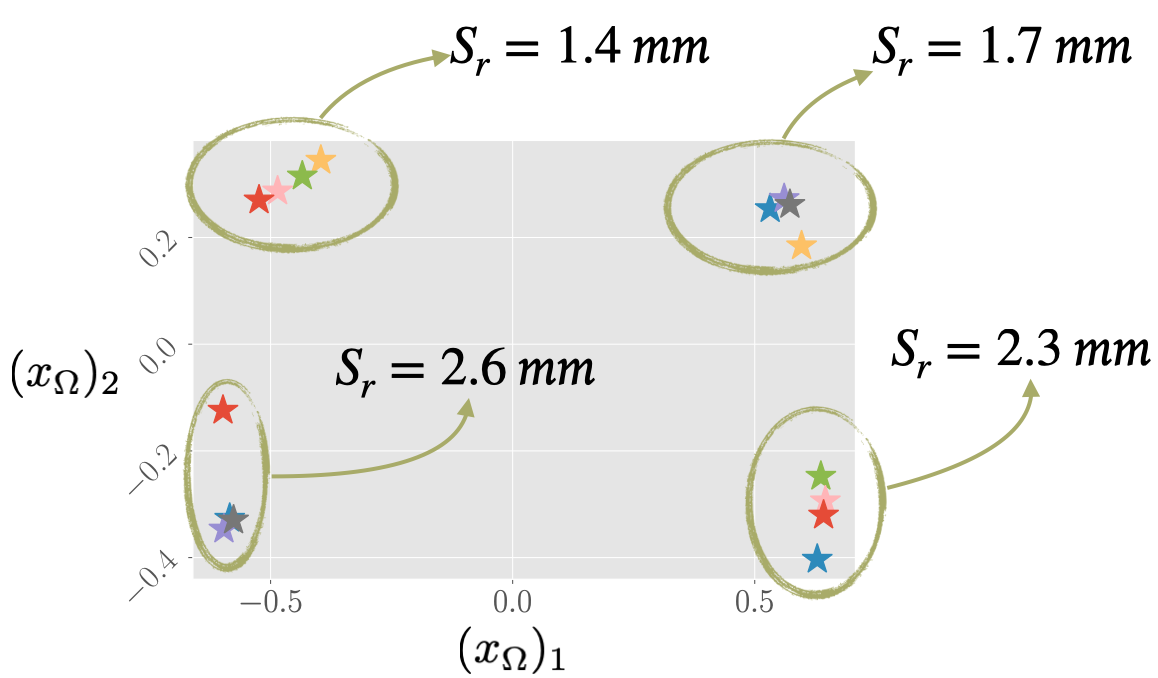}
\label{fig:points-low-dim}
}
\caption{Pair-wise distances plot (normalized scale in [0,1], where blue is 0 and yellow is 1) between 16 geometries and MDS representation in $\bR^2$.}
\end{figure}

\subsection{Reconstruction of synthetic data}
\label{sec:tests-perf}
\paragraph{Definition of the observation space $\Wm(\Omega)$:} For a given $\Omega\in \rG$, we consider a partition of $\Omega = \cup_{i=1}^m \Omega^{\text{voxel}}_i $ into $m$ disjoint subdomains (voxels) $\Omega^{\text{voxel}}_i$. We mimic getting ultrasound images by defining the linear functionals $\ell_i\in L^2(\Omega)$ as
\begin{equation}
\ell_i(u) = \int_{\Omega^{\text{voxel}}_i} u \cdot b ~ \dx,\quad 1\leq i\leq m,
\label{eq:the_measures1D}
\end{equation}
where $b$ is a unitary vector giving the direction of the ultrasound beam. In our case, the plane is chosen to be $z=0$, the ultrasound direction is $b = [0,\sqrt{2}/2,\sqrt{2}/2]$ and the size of voxels is 2.5 mm$^3$. The dimension $m$ of the total number of observations changes slightly between geometries. The geometry with the smallest amount of voxels, i.e., the geometry corresponding to the smaller parameter $S_r$ and maximal $S_l$, is $m=59$. Given that the domain is unknown a-priori, we need to address the construction of the space $W_m(\Omega)=\vspan\{\omega_i\}_{i=1}^m$ during the online phase . The problem of computing the Riesz representers of the measures reads: Find $\{ \omega_i \}_{i=1}^m \in V$ such that
$$
\langle \omega_i, v \rangle_{V(\Omega)} = \int_{\Omega^{\text{voxel}}_i} v \cdot b~\dx \quad \forall v \in V(\Omega),
$$
Since our reconstruction space $V(\Omega)$ is $L^2(\Omega)$, we have that $\omega_i =  \charFun_{\Omega^{\text{voxel}}_i} b$, and the numerical cost of computing the family of representers $\{\omega_i \}_{i=1}^m$ is negligible in our case. In Figure \ref{fig:doppler-img} we give an example of a PDE solution $u$ and its associated synthetic Doppler ultra-sound data $P_\Wm u$.

\begin{figure}[!htbp]
\centering
\subfigure[Example of forward simulation $u \in \cM(\Omega)$]{
\includegraphics[height=1.8cm]{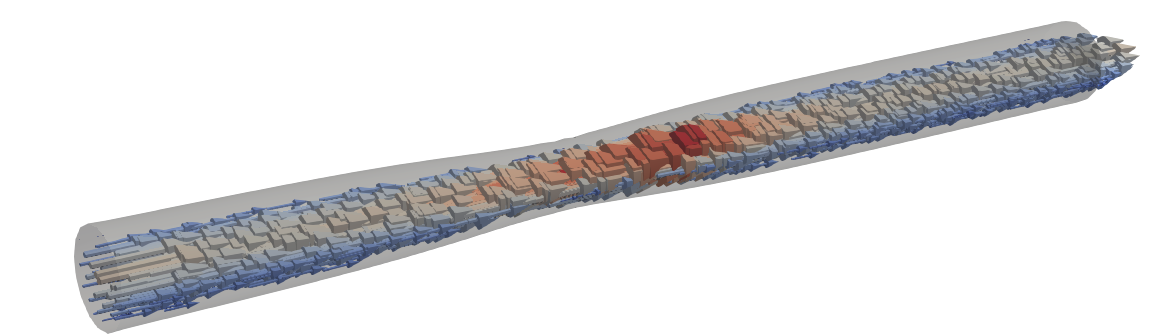}
\includegraphics[height=2cm]{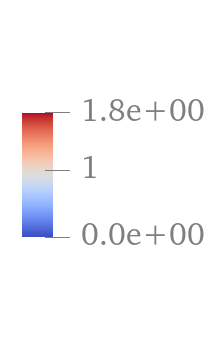}}
\subfigure[Synthetic data $P_{W_m} u$]{
\includegraphics[height=1.8cm]{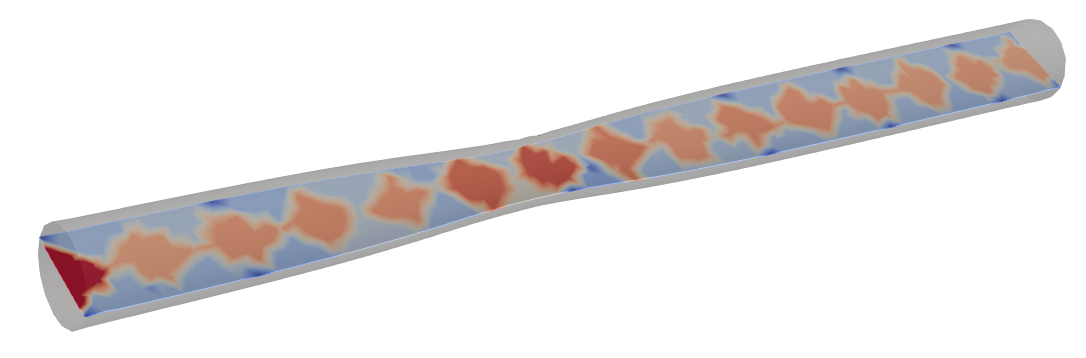}
\includegraphics[height=2cm]{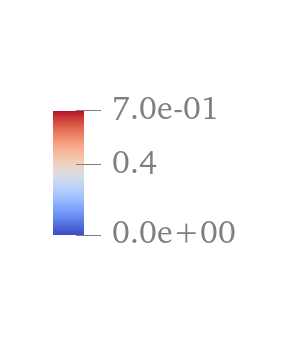}}
\caption{Snapshot in manifold of solutions and its projection in the space $W_m$. The measures emulates Doppler ultrasound data with a transducer steered with an angle of $\pi/4$ respect to the main fluid direction.}
\label{fig:doppler-img}
\end{figure}

\paragraph{Reconstruction:} We test the methodology with $K_{\text{test}} = 16$ test working domains $G_\test = \{\Omega^\test_i\}_{i=1}^{K_{\text{test}}}$ which are taken different from the geometries in $G_\templates$. For each test working domain, we sample $N_{\text{target}} = 16$  target simulations of the governing dynamics in $\cM(\Omega_t^i)$. This yields a total of $50$ snapshots per target due to time marching.

%We verify that the method delivers good reconstructions by choosing either a good or the best available template among the 64 working domains sampled. 

\begin{figure}[!htbp]
\centering
\includegraphics[height=6.5cm,angle=90]{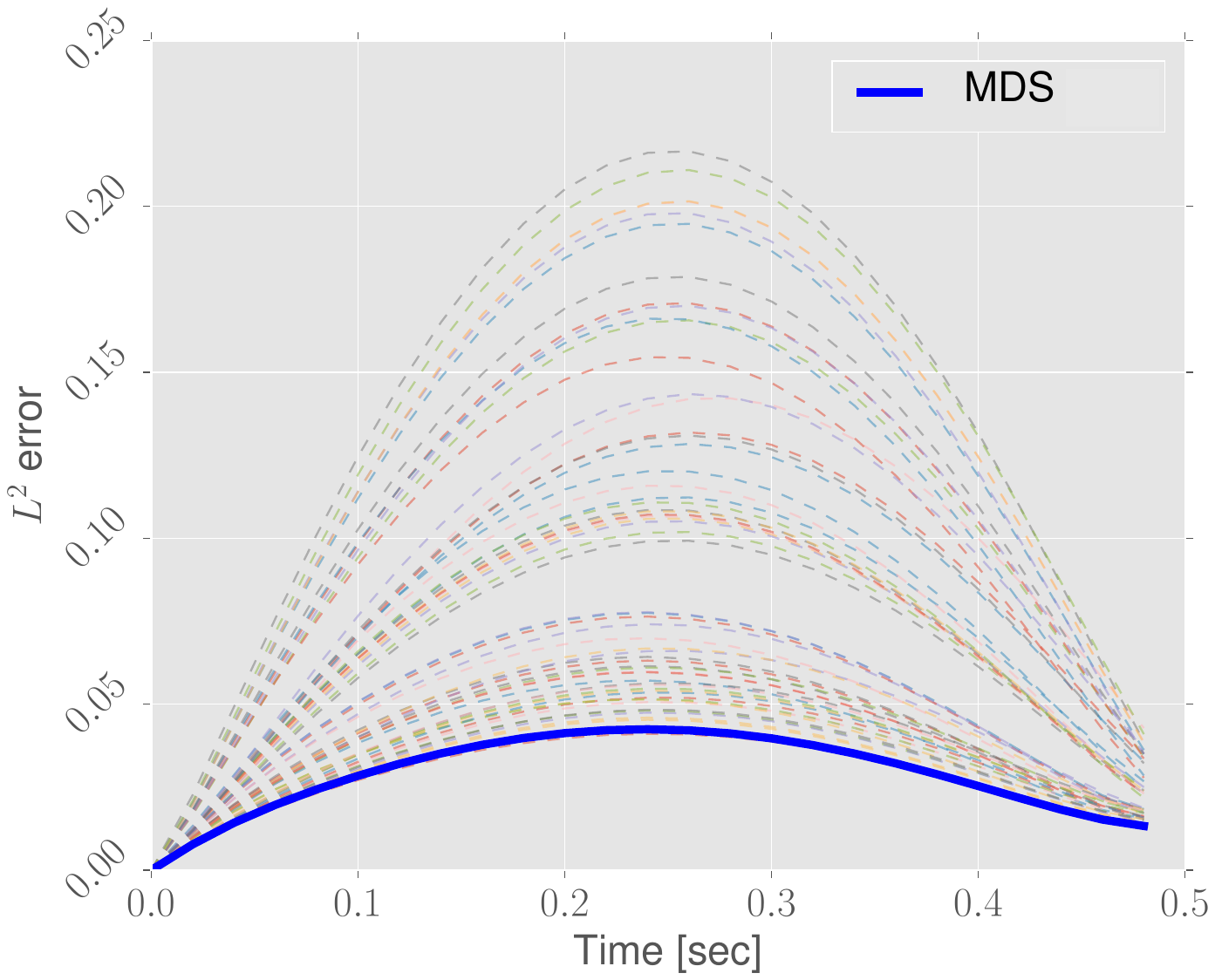}
\includegraphics[height=6.5cm,angle=90]{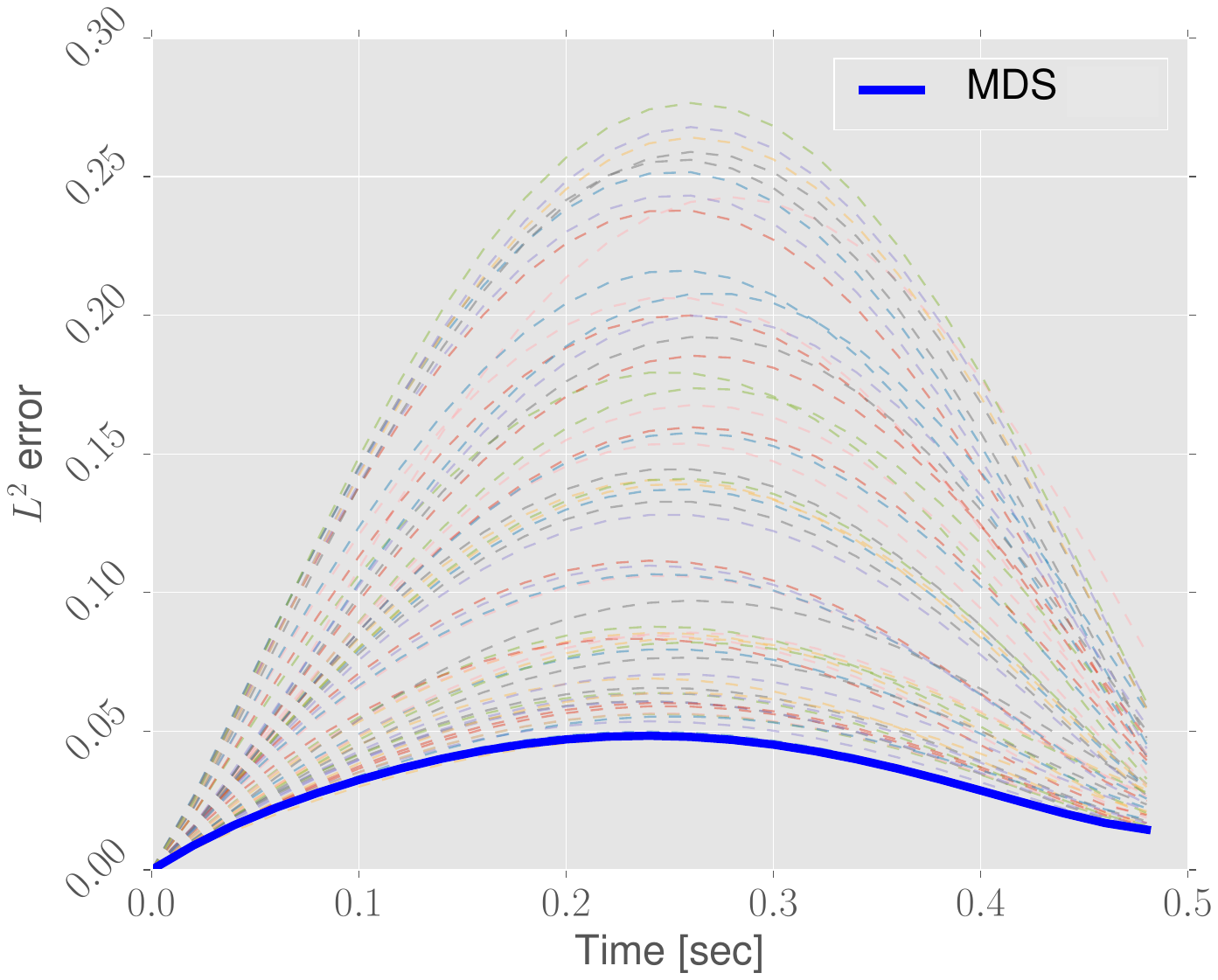}
\includegraphics[height=6.5cm,angle=90]{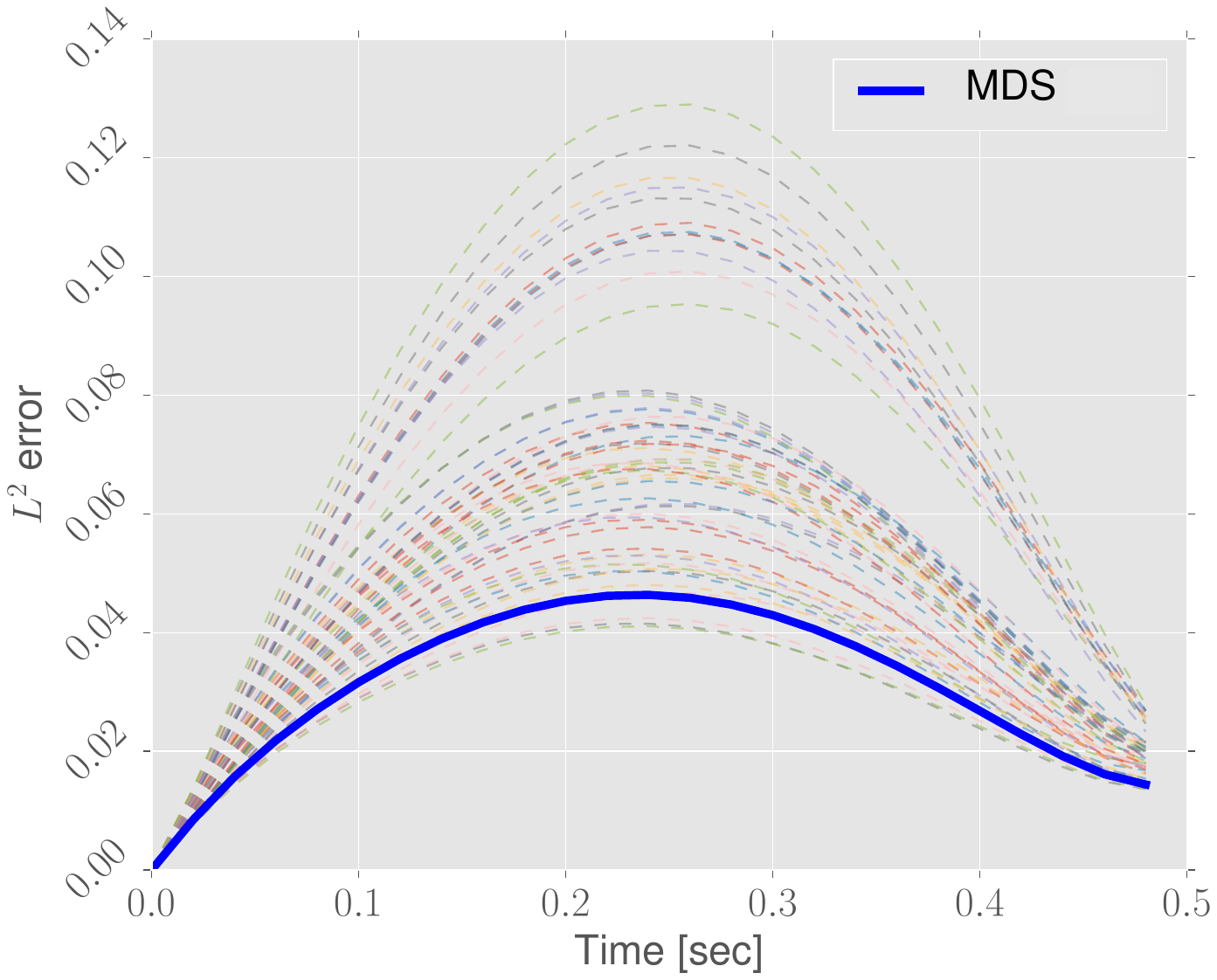}
\includegraphics[height=6.5cm,angle=90]{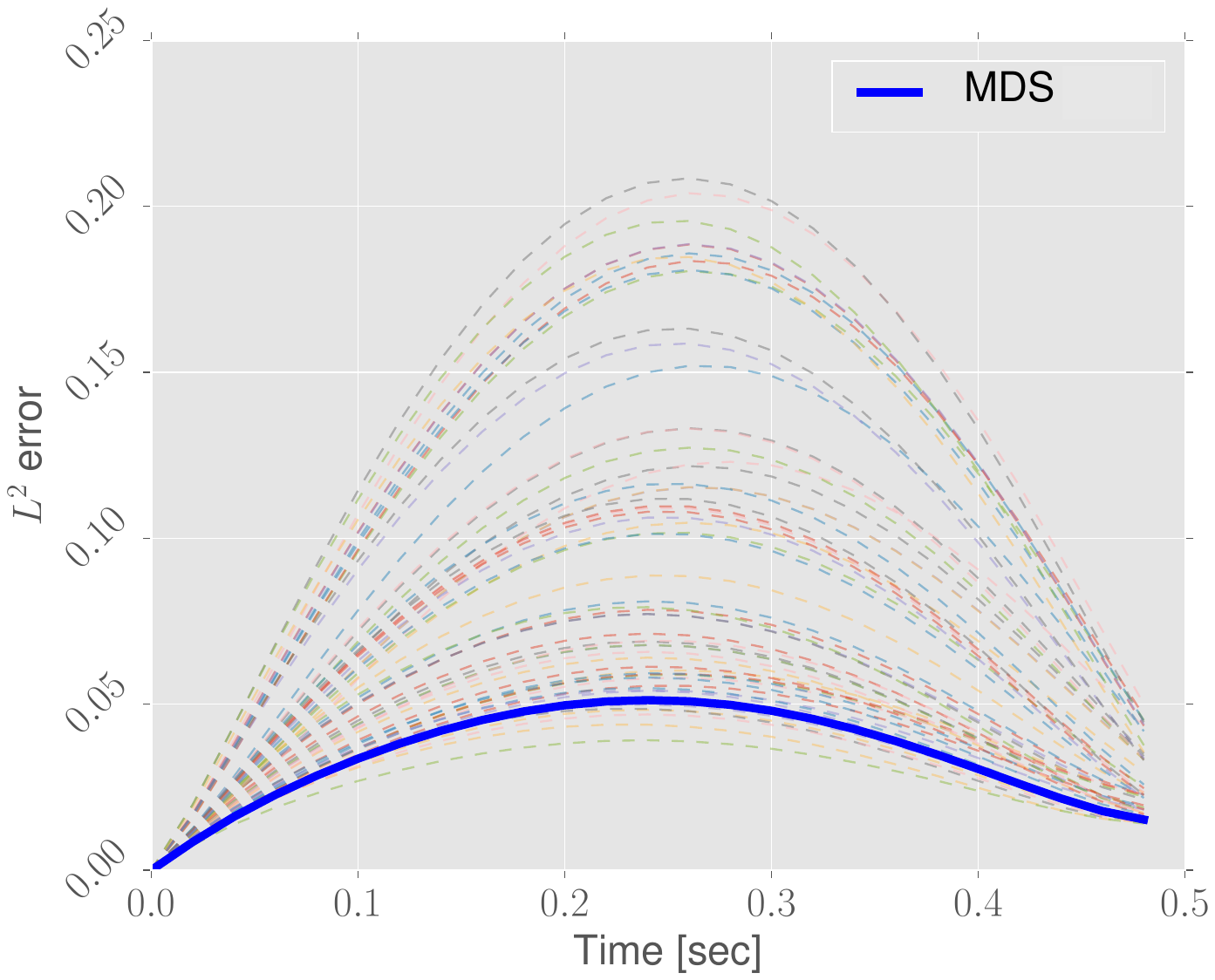}
\caption{The reconstruction error \eqref{eq:error_time} with the best template for 4 different target geometries $\Omega\in G_\test$. The \texttt{Best-Template} methods is able to identify a good or the best template.}
\label{fig:reconstuction_examples}
\end{figure}

We study the performance of our method in terms of relative average reconstruction errors in $L^2$. For a given target geometry $\Omega \in G_\test$, if we reconstruct by transporting linear subspace $V_n(\Omega_t)$ from a given template geometry $\Omega_t\in G_\templates$, the relative error for the $i$-th simulation at time $t$ is defined as
\begin{equation}
e^i_{\Omega_t\to \Omega}(t) = \frac{\norm{u_i(t) - A\(P_{W_m} u(t)\)}}{\(\int_0^T \norm{u_i(t)}^2 ~\dt \)^{1/2} }.
\label{eq:error_time}
\end{equation}
In Figure \ref{fig:reconstuction_examples}, we fix one target geometry $\Omega\in G_\test$ and we show the average error over all simulations $i$, namely
$$
e_{\Omega_t\to \Omega}(t) = \frac{1}{N_\text{target}} \sum_{i=1}^{N_\text{target}}e^i_{\Omega_t\to \Omega}(t).
$$
Each curve depicts the error for each template geometry $\Omega_t \in G_\templates$. The role of the routine \texttt{Best-Template} which we have built in the learning stage is to quickly select the template which will be the most appropriate so that we obtain the most accurate reconstruction results. The selection with our proposed construction yields the error curve which is labeled MDS. We tested several possibilities for the definition of the metric $\rho$ but the one based \eqref{eq:symmetrized} produced systematically the best results, so, for the sake of clarity, we only present the results for this choice. We observe in Figure \ref{fig:reconstuction_examples} that the selection method is near-optimal in the sense that it chooses either a good or the best available template among the 64  template domains. Figure \ref{fig:rec_snap} gives an illustration of the reconstruction of one snapshot with our pipeline.

\begin{figure}[!htbp]
\centering
\subfigure[Target field $u$]{
\includegraphics[height=0.9cm]{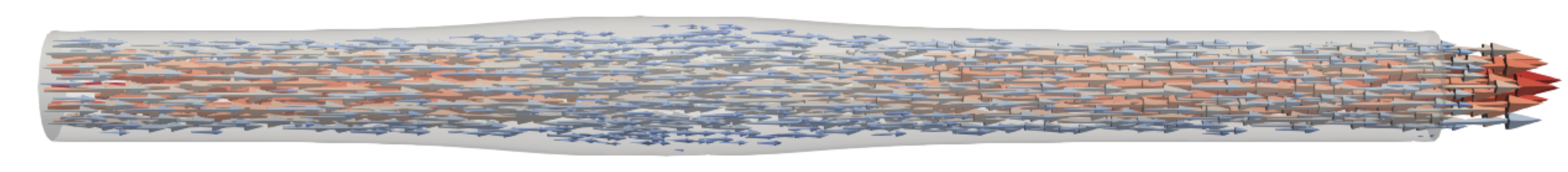}
\includegraphics[height=1.cm]{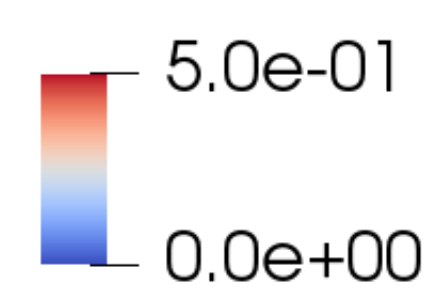}} \\
\subfigure[Reconstructed field $A\(P_{W_m}u\)$]{
\includegraphics[height=1.1cm]{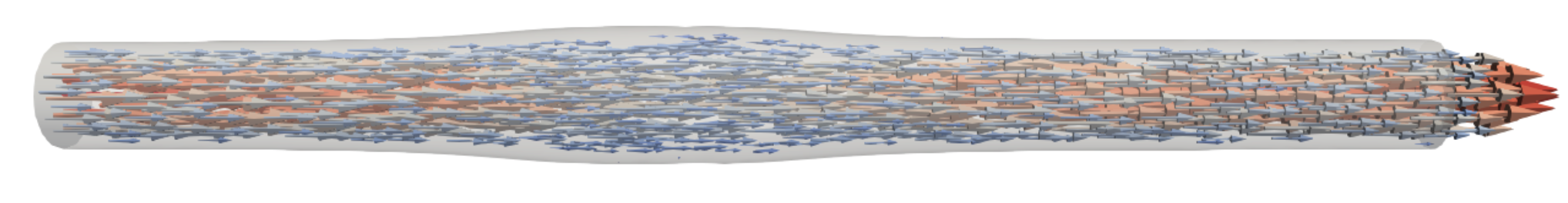}
\includegraphics[height=0.9cm]{./rec12_colorbar.png}} \\
\subfigure[Error field $u - A\(P_{W_m}u\)$]{
\includegraphics[height=0.75cm]{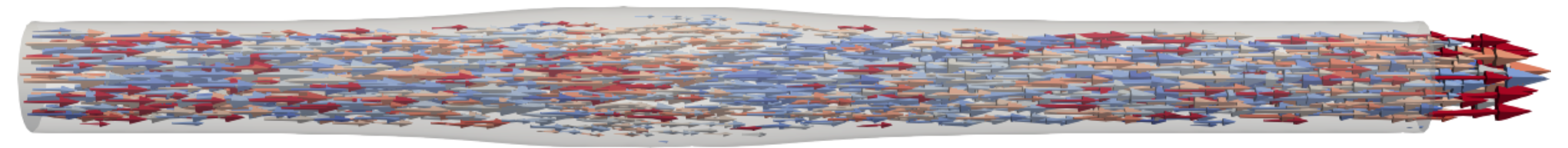}
\includegraphics[height=0.9cm]{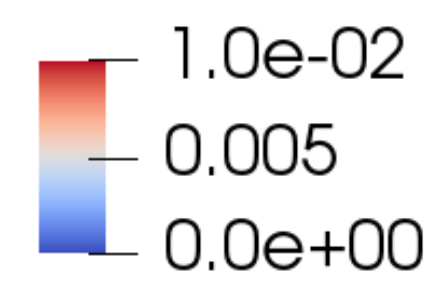}} \\
\caption{Example of field reconstruction for one target snapshot}
\label{fig:rec_snap}
\end{figure}

In addition, it is important to compare with a set of reconstructions on all the $K_{\text{test}}$ test geometries $\Omega_{\text{test}} \in G_{\text{test}}$ with the pre-computed ROMs $V_n(\Omega_{\text{test}})$. We want to quantify the difference between the PBDW algorithm output of $A_{V_n(\Omega_{\text{test}})} ( \omega ) $ and that of $A_{\widehat{V}_n(\Omega)} (\omega)$, for $\omega \in W_m$. We recall that $\widehat{V}_n(\Omega) = \widehat\tau_{\Omega^*_t \to \Omega}( V_n (\Omega^*_t) )$, and $\Omega^*_t = \BT(\Omega_{\text{test}})$.

It is encouraging to observe that the relative error
$$
\max_{i} \frac{ \norm{ A_{V_n(\Omega_{\text{test}})} (\omega_i) - A_{\widehat{V}_n(\Omega)} (\omega_i) }_{L^2} } { \norm{u_{\text{GT}}^{i} }_{L^2} },
$$
is below one percent for a set of 16 ground truth solutions $\{u_{\text{GT}}^{i} \}_{i=1}^{16}$ in each test domain $\Omega_{\text{test}}$, with $\omega_i = P_{W_m} \( u_{\text{GT}}^{i} \)$. 

\section{Conclusion}
We have developed a framework to solve in near-real time state estimation problems for applications that present variations in the spatial domain. For a given target geometry, the reconstruction strategy is based on selecting a relevant linear subspace defined on a template geometry,  which is then transported to the target geometry. The reduced model is chosen among a pool of available reduced models, each one defined on a different template geometry. The model selection strategy is based on a dimensionality reduction technique based on MDS. The technique requires defining an appropriate notion of distance between manifold sets $\cM(\Omega)$ from different geometries $\Omega$. Among the different options for the metric which we have tested in our numerical experiments, the one based on formula \eqref{eq:symmetrized} has produced the best results, and is simple to implement in practice. This choice is backed up by our theoretical analysis from Theorem \ref{thm:bound2}.

The present contribution paves the way for further developments in the field of inverse problems presenting shape variability, especially in the field of biomedical engineering. Future research will be devoted to applying the present methodology to applications with real data, and \new{with more complex geometrical shapes, with possible time-dependency. In principle, increasing the complexity of the domain can be addressed with the same tools as the ones presented in this paper, which are essentially based on LDDMM (for mappings between surfaces) and MDS (for the search of a Euclidean embedding). However, these techniques will inevitably be increasingly challenged for increasingly complex geometries. For LDDMM, we may require working with more surface points to find a good map between surfaces. In turn, this will challenge the underlying optimization task of LDDMM. As for MDS, it may be necessary to increase the number $K$ of available templates geometries to find a good quality embedding, which will increase the computational burden.}

\appendix

\section{Details on the numerical solution of the Stokes equation}
\label{app:stokes-solver}

Using finite elements, we search for the projection coefficients of $u$ and $p$ in the space of piece-wise linear polynomials $[\bP_1(\Omega_h)]^3$ and $\bP_1(\Omega_h)$ respectively. The Lagrange polynomials are considered on $\Omega_h$, a tessellation of $\Omega$ with tetrahedrons of size $h = 0.08$ cms. We don't adopt a new notation for the projection of the states in the polynomial spaces when no confusion arises. Time discretization is done via implicit finite differences using a time step of $\Delta t = 0.02$ seconds. The semi-discrete weak problem to solve for each $u^{n+1}$ reads:
\begin{equation}
\frac{1}{\Delta t} \< u^{n+1}, v \> +  \mu \< \nabla u^{n+1}, \nabla v \> - \< p , \nabla \cdot v \> + \< \nabla \cdot u^{n+1}, q \> + \sum_{\bT} h_{\bT}^2 \< \nabla p, \nabla q \>_{\bT} = \frac{1}{\Delta t} \< u^{n}, v \>,
\label{eq:weak_stokes_discrete}
\end{equation}
$\forall (v,q) \in [H^1(\Omega)]^3 \times L^2(\Omega)$, where $v$ and $q$ are test functions and where $\<\cdot,\cdot\> $ denotes the inner product in $L^2(\Omega)$. In addition, $\<\cdot, \cdot\>_{\bT}$ denotes the $L^2(\Omega_h)$ inner product over a single tetrahedron $\bT$ in $\Omega_h$. The bilinear form concerning this term is a typical stabilization procedure to deal with the saddle point nature of the problem \cite{brezzi1984}. 

The matrix assembly and solution of the monolithic system of equations is done with CPU parallelization via MPI using the software MAD (\cite{galarceThesis}, chapter 5).

\section{\nuevo{Numerical implementation of PBDW}}
\label{app:numerics_pbdw}
This appendix concerns the solution of problem \eqref{eq:pbdw-algo}. We briefly paraphrase and complement the content from \cite[Appendix A]{GLM2021}, where this discussion has been carried out in the context of blood flows.

Let us formulate the problem \eqref{eq:pbdw-algo} in a discrete framework. We adopt the following definitions and considerations:
\begin{itemize}
\item Let $\Omega_h \subset \bR^d$ ($d=2,3$) be a tessellation of the working domain $\Omega$.
\item Let $V_h(\Omega_h)\subseteq V(\Omega)$ be a finite element space on $\Omega_h$. In the following, we work with a finite element basis $\{\cL_i\}_{i=1}^\cN$ for $V_h(\Omega_h)$, where $\cN$ is the number of degrees of freedom of any function in $V_h(\Omega_h)$.
%\item Let $u_h \in V_h(\Omega_h)$ be an approximation of the state $u\in V(\Omega)$ to estimate.
\item Let $M \in \bR^{\cN \times \cN}$ be the mass matrix representing the inner product $\<\cdot, \cdot\>_{V_h(\Omega_h)}$. For example, in the numerical example of section \ref{sec:numerical_example}, we have $V = L^2(\Omega)$, therefore $M_{ij} = \int_{\Omega_h} \cL_i \cL_j \dx$.
\item Let $l=\{l_i\}_{i=1}^m \in \bR^m$ be the vector of measurement observations. Each entry is given by $l_{i} = \< \omega_i, u \>$.
\item Let $\cW \in \bR^{\cN \times m}$ be a matrix where the entry of each column corresponds to the coefficients of the orthonormal Riesz representers $\{\omega_i\}_{i=1^\cN}$ on the basis for $V_h(\Omega_h)$.
\item Let $\{ \rho_i \}_i^n$ be an orthonormal basis for $V_n(\Omega)$, and a matrix $\Phi \in \bR^{\cN \times n}$ where the entry of each column corresponds to the coefficients of the basis for $V_n(\Omega)$ on the basis for $V_h(\Omega_h)$.
\item Let $\Pi \in \bR^{ \cN \times \cN }$ be the matrix representation of the orthogonal projector $P_{V_n^{\perp}}:V\to V_n^\perp$. The matrix is given by $\Pi = I - \Phi \Phi^T M$, where $I$ is an identity matrix of size $\cN$.
\end{itemize}

In this setting, problem \eqref{eq:pbdw-algo} can be written in discrete form as follows: We search for a vector $a^*\in \bR^\cN$, coefficients of $u_h^{*}$ on the basis for $V_h(\Omega_h)$ (i.e. $u_h^* = \sum_i a_i \cL_i$), reconstruction of $u$, such that:
%\om{[Clash of notation: $u^*_h$ was not really defined as the coefficients of the basis but the function itself.]} 
\begin{equation}
\begin{aligned}
a^{*} &= \argmin_{a \in \mathbb{R}^\cN}{ \norm{\Pi a }_{V_h(\Omega_h)}^2 } \\
\text{s.t } & \cW^T M a - l = 0
\end{aligned}
\end{equation}

Consider a Lagrange multiplier $\lambda \in \bR^m$. The problem can be written without restriction as follows:
\begin{equation}
\begin{aligned}
a^{*} &= \argmin_{a \in \mathbb{R}^\cN}{ \norm{\Pi a }_{V_h(\Omega_h)}^2  - \lambda^T \( \cW^T M a - l \) }\\
& = \argmin_{a \in \mathbb{R}^\cN}{ \(\Pi a\)^T M \Pi a - \lambda^T \( \cW^T M a - l \) } \\
& = \argmin_{a \in \mathbb{R}^\cN}{ a^T \Pi M \Pi a  - \lambda^T \( \cW^T M a - l \) } \\
& = \argmin_{a \in \mathbb{R}^\cN}{ a^T M \Pi a  - \lambda^T \( \cW^T M a - l \) },
\end{aligned}
\end{equation}
where the fourth line comes from exploiting the symmetry of $M$ and $\Pi$, and the projector property $\Pi^2 = \Pi$.

Optimality conditions lead to the saddle point problem:
\begin{align}
M \Pi a - M \cW \lambda = 0 \\
\cW^T M a - l = 0.
\label{eq:pbdw_euler_lagrange}
\end{align}
which is a $m+n$ system of equations. Nonetheless, we can reduce the size to $n$ with the orthogonal decomposition $a^{*} = v^*_h + \eta^*_h$ where
\begin{equation}
v^*_h = \( I - \Pi \) a^*, \quad \eta^*_h = \Pi  a^{*}.
\end{equation}
%\om{[With this definition it seems that $v^*_h\in V_n^T$, and this is the opposte of what we want.]}

In addition, consider the Gramian matrix defined as $\bG = \cW^T M \Phi$ (of dimensions $m \times n$) and the expansion of $v^*_h = \Phi c$, where $c \in \bR^n$.
%Thus, multiplying \eqref{eq:pbdw_euler_lagrange} by $\bG^T$ leads to the normal equations
Thus, \eqref{eq:pbdw_euler_lagrange} leads to the normal equations
\begin{equation}
\bG^T \bG c = \bG^T l.
\label{eq:pbdw_normal_equations}
\end{equation}
Assumption $\beta(V_n, W_m)>0$ guarantees that the matrix $\bG^T\bG$ is invertible. Therefore the core operation to solve the original optimization problem \eqref{eq:pbdw-algo} consists in solving an $n \times n$ system of normal equations at online phase to compute $v^*_h$, after which it follows that the state is fully computed doing $\eta^*_h = \cW \cW^T M v^*_h - \omega$.

The algorithm (and the cost of each step) can be thus summarized as follows:
\begin{itemize}
\item Offline:
\begin{enumerate}
\item Compute the training manifold ( $\ord(\cN^\alpha N_s)$). $\alpha$ depends on the linear solver, tipically $\alpha = 2$.
\item Compute the POD for $V_n(\Omega)$ ($\ord(\cN N_s^2)$).
\item Compute the observation space ($\ord(\cN^\alpha m)$).
\item Compute the matrix $\bG$ ($\ord(2\cN^2 m n)$).
\item Compute the matrix $\bG^T \bG$ ($\ord(mn^2)$).
\item Compute Cholesky factorization $\bG^T \bG = R^T R$ ($\ord(n^3/3$)).
\end{enumerate}
\item Online:
\begin{enumerate}
\item Solve the lower-triangular system $R^T w = \bG^T l$ for $w$ by backward substitution ($\ord(n^2$)).
\item Solve the upper-trianguler system $Rc = w$ for $c$ by backward substitution ($\ord(n^2$)).
\item Compute $\eta^*$ and the solution $v^* + \eta^*$ ($\ord(2(\cN m n))$).
\end{enumerate}
\end{itemize}
The reader must recall that this is the scheme for the single domain problem ($K = 1$). For the multi domain problem ($K > 1$), the online stage starts from item 3 of offline phase onwards, since the observation space is supposed to be unkown.

\bibliographystyle{unsrt}
\bibliography{GLM2021.bib}

\begin{thebibliography}{10}

\bibitem{AZF2012}
D.~Amsallem, M.~J. Zahr, and C.~Farhat.
\newblock Nonlinear model order reduction based on local reduced-order bases.
\newblock {\em International Journal for Numerical Methods in Engineering},
  92(10):891--916, 2012.

\bibitem{Welper2017-TSI}
Gerrit Welper.
\newblock Interpolation of functions with parameter dependent jumps by
  transformed snapshots.
\newblock {\em SIAM Journal on Scientific Computing}, 39(4):A1225--A1250, 2017.

\bibitem{ELMV2020}
V.~Ehrlacher, D.~Lombardi, O.~Mula, and F.-X. Vialard.
\newblock Nonlinear model reduction on metric spaces. application to
  one-dimensional conservative pdes in wasserstein spaces.
\newblock {\em ESAIM M2AN}, 54(6):2159--2197, 2020.

\bibitem{BCDGJP2021}
Bonito A., Cohen A., R.~DeVore, D.~Guignard, P.~Jantsch, and G.~Petrova.
\newblock Nonlinear methods for model reduction.
\newblock {\em ESAIM: Mathematical Modelling and Numerical Analysis},
  55(2):507--531, 2021.

\bibitem{MPPY2015}
Y.~Maday, A.~T. Patera, J.~D. Penn, and M.~Yano.
\newblock A parameterized-background data-weak approach to variational data
  assimilation: formulation, analysis, and application to acoustics.
\newblock {\em International Journal for Numerical Methods in Engineering},
  102(5):933--965, 2015.

\bibitem{MMT2016}
Y.~Maday, O.~Mula, and G.~Turinici.
\newblock Convergence analysis of the {G}eneralized {E}mpirical {I}nterpolation
  {M}ethod.
\newblock {\em SIAM Journal on Numerical Analysis}, 54(3):1713--1731, 2016.

\bibitem{BCDDPW2017}
P.~Binev, A.~Cohen, W.~Dahmen, R.~DeVore, G.~Petrova, and P.~Wojtaszczyk.
\newblock Data assimilation in reduced modeling.
\newblock {\em SIAM/ASA Journal on Uncertainty Quantification}, 5(1):1--29,
  2017.

\bibitem{Taddei2017}
T.~Taddei.
\newblock An adaptive parametrized-background data-weak approach to variational
  data assimilation.
\newblock {\em ESAIM: Mathematical Modelling and Numerical Analysis},
  51(5):1827--1858, 2017.

\bibitem{BCMN2018}
P.~Binev, A.~Cohen, O.~Mula, and J.~Nichols.
\newblock Greedy algorithms for optimal measurements selection in state
  estimation using reduced models.
\newblock {\em SIAM/ASA Journal on Uncertainty Quantification},
  6(3):1101--1126, 2018.

\bibitem{GMMT2019}
H.~{Gong}, Y.~{Maday}, O.~{Mula}, and T.~{Taddei}.
\newblock {PBDW} method for state estimation: error analysis for noisy data and
  nonlinear formulation.
\newblock {\em arXiv e-prints}, page arXiv:1906.00810, Jun 2019.

\bibitem{CDDFMN2020}
A.~Cohen, W.~Dahmen, R.~DeVore, J.~Fadili, O.~Mula, and J.~Nichols.
\newblock {Optimal reduced model algorithms for data-based state estimation}.
\newblock {\em SIAM Journal on Numerical Analysis}, 58(6):3355--3381, 2020.

\bibitem{CDMN2020}
A.~Cohen, W.~Dahmen, O.~Mula, and J.~Nichols.
\newblock Nonlinear reduced models for state and parameter estimation.
\newblock {\em arXiv:2009.02687}, 2020.

\bibitem{guibert2014}
R.~Guibert, K.~Mcleod, A.~Caiazzo, T.~Mansi, M.A. Fern{\'a}ndez, M.~Sermesant,
  X.~Pennec, I.E. Vignon-Clementel, Y.~Boudjemline, and J.F. Gerbeau.
\newblock Group-wise construction of reduced models for understanding and
  characterization of pulmonary blood flows from medical images.
\newblock {\em Medical image analysis}, 18(1):63--82, 2014.

\bibitem{davies2008}
R.~Davies, C.~Twining, and C.~Taylor.
\newblock {\em Statistical models of shape: Optimisation and evaluation}.
\newblock Springer Science \& Business Media, 2008.

\bibitem{maury2017}
A.~Maury, G.~Allaire, and F.~Jouve.
\newblock Shape optimisation with the level set method for contact problems in
  linearised elasticity.
\newblock {\em The SMAI journal of computational mathematics}, 3:249--292,
  2017.

\bibitem{de2008}
F.~De~Gournay, G.~Allaire, and F.~Jouve.
\newblock Shape and topology optimization of the robust compliance via the
  level set method.
\newblock {\em ESAIM: Control, Optimisation and Calculus of Variations},
  14(1):43--70, 2008.

\bibitem{colton1998}
D.L. Colton and R.~Kress.
\newblock {\em Inverse acoustic and electromagnetic scattering theory},
  volume~93.
\newblock Springer, 1998.

\bibitem{bookstein2018}
F.L. Bookstein.
\newblock {\em A course in morphometrics for biologists: geometry and
  statistics for studies of organismal form}.
\newblock Cambridge University Press, 2018.

\bibitem{mitteroecker2009}
P.~Mitteroecker and P.~Gunz.
\newblock Advances in geometric morphometrics.
\newblock {\em Evolutionary Biology}, 36(2):235--247, 2009.

\bibitem{lovgren2006}
A.~E. L{\o}vgren, Y.~Maday, and E.M. R{\o}nquist.
\newblock The reduced basis element method for fluid flows.
\newblock In {\em Analysis and Simulation of Fluid Dynamics}, pages 129--154.
  Springer, 2006.

\bibitem{rozza2013}
G.~Rozza, D.B.P. Huynh, and A.~Manzoni.
\newblock Reduced basis approximation and a posteriori error estimation for
  {S}tokes flows in parametrized geometries: roles of the inf-sup stability
  constants.
\newblock {\em Numerische Mathematik}, 125(1):115--152, 2013.

\bibitem{manzoni2017eff}
A.~Manzoni and F.~Negri.
\newblock Efficient reduction of {PDE}s defined on domains with variable shape.
\newblock In {\em Model Reduction of Parametrized Systems}, pages 183--199.
  Springer, 2017.

\bibitem{dal2019hyper}
N.~Dal~Santo and A.~Manzoni.
\newblock Hyper-reduced order models for parametrized unsteady
  {N}avier-{S}tokes equations on domains with variable shape.
\newblock {\em Advances in Computational Mathematics}, 45(5):2463--2501, 2019.

\bibitem{chamoin2019}
L.~Chamoin and H.P. Thai.
\newblock Certified real-time shape optimization using isogeometric analysis,
  {PGD} model reduction, and a posteriori error estimation.
\newblock {\em International Journal for Numerical Methods in Engineering},
  119(3):151--176, 2019.

\bibitem{garotta2020}
Fabrizio Garotta, Nicola Demo, Marco Tezzele, Massimo Carraturo, Alessandro
  Reali, and Gianluigi Rozza.
\newblock Reduced order isogeometric analysis approach for pdes in parametrized
  domains.
\newblock In {\em Quantification of Uncertainty: Improving Efficiency and
  Technology}, pages 153--170. Springer, 2020.

\bibitem{sevilla2020}
R.~Sevilla, S.~Zlotnik, and A.~Huerta.
\newblock Solution of geometrically parametrised problems within a {CAD}
  environment via model order reduction.
\newblock {\em Computer methods in applied mechanics and engineering},
  358:112631, 2020.

\bibitem{piegl1996nurbs}
L.~Piegl and W.~Tiller.
\newblock {\em The {NURBS} book}.
\newblock Springer Science \& Business Media, 1996.

\bibitem{salmoiraghi2018}
F.~Salmoiraghi, A.~Scardigli, H.~Telib, and G.~Rozza.
\newblock Free-form deformation, mesh morphing and reduced-order methods:
  enablers for efficient aerodynamic shape optimisation.
\newblock {\em International Journal of Computational Fluid Dynamics},
  32(4-5):233--247, 2018.

\bibitem{berkooz1993proper}
G.~Berkooz, P.~Holmes, and J.L. Lumley.
\newblock The proper orthogonal decomposition in the analysis of turbulent
  flows.
\newblock {\em Annual review of fluid mechanics}, 25(1):539--575, 1993.

\bibitem{sirovich1987turbulence}
L.~Sirovich.
\newblock Turbulence and the dynamics of coherent structures. {I}. coherent
  structures.
\newblock {\em Quarterly of applied mathematics}, 45(3):561--571, 1987.

\bibitem{rathinam2003new}
M.~Rathinam and L.R. Petzold.
\newblock A new look at proper orthogonal decomposition.
\newblock {\em {SIAM} Journal on Numerical Analysis}, 41(5):1893--1925, 2003.

\bibitem{karatzas2019}
E.N. Karatzas, G.~Stabile, L.~Nouveau, G.~Scovazzi, and G.~Rozza.
\newblock A reduced basis approach for {PDE}s on parametrized geometries based
  on the shifted boundary finite element method and application to a {S}tokes
  flow.
\newblock {\em Computer Methods in Applied Mechanics and Engineering},
  347:568--587, 2019.

\bibitem{akkari2019}
N.~Akkari, F.~Casenave, and D.~Ryckelynck.
\newblock A novel {G}appy reduced order method to capture non-parameterized
  geometrical variation in fluid dynamics problems.
\newblock 2019.

\bibitem{karatzas2020}
E.N. Karatzas, F.~Ballarin, and G.~Rozza.
\newblock Projection-based reduced order models for a cut finite element method
  in parametrized domains.
\newblock {\em Computers \& Mathematics with Applications}, 79(3):833--851,
  2020.

\bibitem{stabile2020}
G.~Stabile, M.~Zancanaro, and G.~Rozza.
\newblock Efficient geometrical parametrization for finite-volume-based reduced
  order methods.
\newblock {\em International Journal for Numerical Methods in Engineering},
  121(12):2655--2682, 2020.

\bibitem{forti2014}
D.~Forti and G.~Rozza.
\newblock Efficient geometrical parametrisation techniques of interfaces for
  reduced-order modelling: application to fluid--structure interaction coupling
  problems.
\newblock {\em International Journal of Computational Fluid Dynamics},
  28(3-4):158--169, 2014.

\bibitem{hess2014}
M.W. Hess and P.~Benner.
\newblock A reduced basis method for microwave semiconductor devices with
  geometric variations.
\newblock {\em COMPEL: The International Journal for Computation and
  Mathematics in Electrical and Electronic Engineering}, 2014.

\bibitem{taddei2020}
T.~Taddei.
\newblock A registration method for model order reduction: data compression and
  geometry reduction.
\newblock {\em SIAM Journal on Scientific Computing}, 42(2):A997--A1027, 2020.

\bibitem{taddei2020b}
T.~Taddei and L.~Zhang.
\newblock Space-time registration-based model reduction of parameterized
  one-dimensional hyperbolic {PDE}s.
\newblock {\em arXiv preprint arXiv:2004.06693}, 2020.

\bibitem{GGLM2021}
F.~Galarce, J.F. Gerbeau, D.~Lombardi, and O.~Mula.
\newblock Fast reconstruction of 3{D} blood flows from doppler ultrasound
  images and reduced models.
\newblock {\em Computer Methods in Applied Mechanics and Engineering},
  375:113559, 2021.

\bibitem{CD2015acta}
A.~Cohen and R.~DeVore.
\newblock Approximation of high-dimensional parametric {PDE}s.
\newblock {\em Acta Numerica}, 24:1--159, 2015.

\bibitem{BMPPT2012}
A.~Buffa, Y.~Maday, A.~T. Patera, C.~Prud'homme, and G.~Turinici.
\newblock A priori convergence of the greedy algorithm for the parametrized
  reduced basis method.
\newblock {\em ESAIM: Mathematical Modelling and Numerical Analysis},
  46(3):595--603, 2012.

\bibitem{RHP2007}
G.~Rozza, D.~B.~P. Huynh, and A.~T. Patera.
\newblock Reduced basis approximation and a posteriori error estimation for
  affinely parametrized elliptic coercive partial differential equations.
\newblock {\em Archives of Computational Methods in Engineering}, 15(3):1, Sep
  2007.

\bibitem{CDS2011}
A.~Cohen, R.~DeVore, and C.~Schwab.
\newblock Analytic regularity and polynomial approximation of parametric and
  stochastic elliptic {PDE}'s.
\newblock {\em Analysis and Applications}, 09(01):11--47, 2011.

\bibitem{GLM2021}
F.~Galarce, D.~Lombardi, and O.~Mula.
\newblock Reconstructing haemodynamics quantities of interest from doppler
  ultrasound imaging.
\newblock {\em Int. J. Numer. Meth. Biomedical Eng.}, 2021.

\bibitem{lddmm}
Faisal Beg, Michael Miller, Alain Trouvé, and Laurent Younes.
\newblock Computing large deformation metric mappings via geodesic flows of
  diffeomorphisms.
\newblock {\em International Journal of Computer Vision}, 65:139--157, 2005.

\bibitem{MN1995}
H.~Murase and S.~K. Nayar.
\newblock Visual learning and recognition of 3{D} objects from appearance.
\newblock {\em International journal of computer vision}, 14(1):5--24, 1995.

\bibitem{TDL2000}
J.~B. Tenenbaum, V.~De~Silva, and J.~C. Langford.
\newblock A global geometric framework for nonlinear dimensionality reduction.
\newblock {\em science}, 290(5500):2319--2323, 2000.

\bibitem{DG2005}
D.~L. Donoho and C.~Grimes.
\newblock Image manifolds which are isometric to euclidean space.
\newblock {\em Journal of mathematical imaging and vision}, 23(1):5--24, 2005.

\bibitem{GGKC2020}
B.~Ghojogh, A.~Ghodsi, F.~Karray, and M.~Crowley.
\newblock Multidimensional scaling, sammon mapping, and isomap: Tutorial and
  survey.
\newblock {\em arXiv preprint arXiv:2009.08136}, 2020.

\bibitem{JMLR:v22:20-275}
B.~Charlier, J.~Feydy, J.A. Glaunès, F.D. Collin, and G.~Durif.
\newblock Kernel operations on the {GPU}, with autodiff, without memory
  overflows.
\newblock {\em Journal of Machine Learning Research}, 22(74):1--6, 2021.

\bibitem{ciarlet1988}
P.~Ciarlet.
\newblock {\em Mathematical Elasticity, vol. I, Studies in Mathematics and its
  Applications}, volume~20.
\newblock North-Holland Publishing Co., 1988.

\bibitem{brezzi1984}
F.~Brezzi and J.~Pitkaranta.
\newblock On the stabilization of finite element approximations of the {S}tokes
  equations.
\newblock {\em Efficient Solutions of Elliptic Systems}, 1984.

\bibitem{galarceThesis}
F.~Galarce.
\newblock {\em Inverse problems in hemodynamics. Fast estimation of blood flows
  from medical data.}
\newblock PhD thesis, INRIA Paris \& Laboratoire Jacques-Louis Lions. Sorbonne
  Universit\'e, 2021.

\end{thebibliography}

\end{document}